\documentclass[11pt,reqno]{amsart}
\usepackage[utf8]{inputenc}

\usepackage{amsfonts}
\usepackage{amsmath,amscd}
\usepackage{amssymb}
\usepackage{tikz-cd}
\usepackage{amsthm}
\usepackage{mathtools}
\usepackage{extpfeil}
\usepackage{cases}

\usepackage{galois}
\usepackage{chemarrow}
\usepackage[all]{xy}
\usepackage{graphicx}
\usepackage[colorlinks,
linkcolor=blue,
linktoc=all,
anchorcolor=black,
citecolor=red
]{hyperref}	
\usepackage{mathrsfs}
\usepackage{extarrows}
\usepackage{texnames}
\usepackage[nameinlink,capitalize]{cleveref}

\def\Z{{\rm \mathbb{Z}}}
\def\N{{\rm \mathbb{N}}}

\def\C{{\rm \mathbb{C}}}
\def\R{{\rm \mathbb{R}}}

\def\rk{{\rm rk}}

\def\SL{{\rm SL}}

\renewcommand{\bar}[1]{\overline{#1}}

\renewcommand{\epsilon}{\varepsilon}
\renewcommand{\hat}[1]{\widehat{#1}}

\renewcommand{\tilde}[1]{\widetilde{#1}}

\DeclarePairedDelimiter{\set}{\{}{\}}
\DeclarePairedDelimiter{\ang}{\langle}{\rangle}

\DeclareMathOperator{\im}{im}

\newcommand{\CP}{\mathbb{CP}}

\newcommand{\bbH}{\mathbb{H}}

\newcommand{\calB}{\mathcal{B}}
\newcommand{\calC}{\mathcal{C}}

\newcommand{\calL}{\mathcal{L}}
\newcommand{\calM}{\mathcal{M}}
\newcommand{\calO}{\mathcal{O}}
\newcommand{\calP}{\mathcal{P}}
\newcommand{\calQ}{\mathcal{Q}}

\newcommand{\frakb}{\mathfrak{b}}
\newcommand{\frakc}{\mathfrak{c}}

\newcommand{\frakg}{\mathfrak{g}}
\newcommand{\frakh}{\mathfrak{h}}

\newcommand{\frakl}{\mathfrak{l}}
\newcommand{\frakm}{\mathfrak{m}}

\newcommand{\frakp}{\mathfrak{p}}

\newcommand{\fraks}{\mathfrak{s}}

\newcommand{\fraku}{\mathfrak{u}}

\newcommand{\fraksl}{\mathfrak{sl}}

\newcommand{\Fr}{\mathrm{Fr}}

\newcommand{\PSL}{\mathrm{PSL}}

\newcommand{\SO}{\mathrm{SO}}
\newcommand{\Sp}{\mathrm{Sp}}

\newcommand{\U}{\mathrm{U}}
\newcommand{\sfB}{\mathsf{B}}
\newcommand{\del}{\partial}
\newcommand{\delbar}{\bar{\partial}}

\DeclareMathOperator{\ad}{ad}
\DeclareMathOperator{\Ad}{Ad}

\theoremstyle{plain}

\newtheorem{theorem}{Theorem}[section]

\newtheorem{proposition/example}[theorem]{Proposition/Example}
\newtheorem{proposition}[theorem]{Proposition}

\newtheorem{lemma}[theorem]{Lemma}

\theoremstyle{definition}
\newtheorem{definition}[theorem]{Definition}
\newtheorem{remark}[theorem]{Remark}

\newtheorem{conjecture/question}[theorem]{Conjecture/Question}

\newtheorem{remark/definition}[theorem]{Remark/Definition}
\newtheorem{definition/notation}[theorem]{Definition/Notation}

\numberwithin{equation}{section}

\begin{document}
\title[Conformal limits in Cayley components]{Conformal limits in Cayley components and $\Theta$-positive opers}

\author{Georgios Kydonakis and Mengxue Yang}


\date{}

\begin{abstract}
  We study Gaiotto's conformal limit for the $G^{\mathbb{R}}$-Hitchin equations, when $G^{\mathbb{R}}$ is a simple real Lie group admitting a $\Theta$-positive structure. We identify a family of flat connections coming from certain solutions to the equations for which the conformal limit exists and admits the structure of an oper. We call this new class of opers appearing in the conformal limit $\Theta$-positive opers. The two families involved are parameterized by the same base space. This space is a generalization of the base of Hitchin's integrable system in the case when the structure group is a split real group. 
\end{abstract}

\maketitle

\section{Introduction}

For a compact Riemann surface $X$ and a complex simple Lie group $G$, the $G$-character variety parameterizes conjugacy classes of fundamental group representations $\rho: \pi_1(\Sigma) \to G$, where $\Sigma$ is the underlying topological surface of $X$. There exists a unique embedding up to conjugation $\iota: \mathrm{SL}_2(\mathbb{C}) \to G$, called the principal embedding, determined by a principal nilpotent element in the Lie algebra $\mathfrak{g}$ of $G$. 
Hitchin in \cite{Hit92} used the restriction of this embedding to $\mathrm{SL}_2(\mathbb{R})$, in which case $\iota(\mathrm{SL}_2(\mathbb{R}))$ is contained in the split real form $G^{\mathbb{R}} \subset G$, in order to demonstrate the existence of a connected component in the character variety, now called the Hitchin component, that is topologically trivial and contains the Teichm\"{u}ller space of $\Sigma$. 
Using the nonabelian Hodge correspondence, the Hitchin component in its original realization was discovered as a set of $G$-Higgs bundles parameterized by a certain vector space of holomorphic differentials over $X$. The $G$-Higgs bundles in this family have been constructed as a special section in the entire moduli space of polystable $G$-Higgs bundles over $X$. 

For a fixed point in the base space of Hitchin's integrable system, one can associate to the Higgs bundles in the Hitchin section, a family of flat $G$-connections in parameters $\zeta \in \mathbb{C}^*$ and $R \in \mathbb{R}^+$. It was shown in \cite{Opers16} that in the scaling limit as $R \to 0$ and $\zeta = \hbar R$ for fixed $\hbar \in \mathbb{C}^*$, this family converges to a flat connection giving the structure of a $G$-oper. This limiting oper is parameterized by the same point on the base of Hitchin's integrable system. In the case when $G=\mathrm{SL}_2(\mathbb{C})$, for instance, a $G$-oper can be alternatively thought of as any of the following: a global Schr\"{o}dinger operator, a complex projective structure on $X$, or a globally defined quantum curve on $X$. 

Gaiotto studied in \cite{Gaiotto} the TBA-like integral equations from \cite{GMN} illustrating the metric on the moduli spaces $\mathcal{M}$ of vacua of four-dimensional $\mathcal{N}=2$ supersymmetric gauge theories compactified on a circle. For theories of class $\mathcal{S}$, the moduli spaces $\mathcal{M}$ coincide roughly with the moduli spaces of solutions to Hitchin's equations. When restricted to the Hitchin section, the solutions gain an extra symmetry which allows one to show explicitly how such solutions would go to opers in the conformal limit. 

\vspace{2mm}

Our objective in the present article is to introduce two sorts of families of flat $G$-connections over $X$ for which a similar passage to opers via the conformal limit  exists, in more general situations than the one outlined above when confining to the Hitchin section and a split real Lie group $G$. Indeed, the Hitchin component is a particular case of what is now called a \emph{higher Teichm\"{u}ller space}. The notion of a \emph{$\Theta$-positive structure} for a simple real Lie group was introduced by Guichard and Wienhard in \cite {GW} in order to propose a unified approach to the study of connected components in moduli spaces of $G$-Higgs bundles or, equivalently, $G$-character varieties, which share essential geometric properties to the classical Teichm\"{u}ller space when $G=\mathrm{PSL}_2(\mathbb{R})$.

The $\Theta$-positive structure of Guichard and Wienhard is a generalization of Lusztig's total positivity condition. The notion is given in terms of properties of the Lie algebra of parabolic subgroups in $G^{\mathbb{R}}$ defined by a subset $\Theta$ of simple positive roots. An important classification theorem then signifies precisely which families of simple real Lie groups $G^{\mathbb{R}}$ admit a $\Theta$-positive structure. For each of these cases in the classification, there exists a natural $\mathfrak{sl}_2(\mathbb{R})$-triple in the real Lie algebra and a canonical real form associated to it, which generalize  the principal $\mathfrak{sl}_2$-triple and the real form of Kostant used for the construction of the Hitchin section. 

Alternatively, these triples can be approached via the weight modules of the Lie algebra $\mathfrak{g}$. In \cite{Cayley}, the notion of a magical $\mathfrak{sl}_2$-triple was introduced and a canonical real form associated to it was obtained; these triples are in bijection with the $\Theta$-principal $\mathfrak{sl}_2(\mathbb{R})$-triples of Guichard and Wienhard. Using this new Lie theoretic machinery,  distinguished connected components in the moduli space of polystable $G^{\mathbb{R}}$-Higgs bundles were exhibited in \cite{Cayley}, called \emph{Cayley components}, analogously to Hitchin's parameterization using the base space of holomorphic differentials. 

It is this Lie theoretic machinery that we use in the present work in order to get our distinguished 2-parameter families of flat $G$-connections for the Lie groups $G$ admitting a $\Theta$-positive structure. Namely, we consider inside a Cayley component a \emph{Slodowy slice} at a uniformizing fixed point of a certain base space, and take the 2-parameter family built from the corresponding solutions to the Hitchin equations. We show that their conformal limit exists and is a point in a \emph{Slodowy slice} of certain $G$-opers at a uniformizing fixed point. These $G$-opers are characterized by the parabolic subgroup $P_{\Theta}$ associated to a $\Theta$-principal $\mathfrak{sl}_2$-triple, thus we call them here \emph{$\Theta$-positive opers}. The precise description of these Slodowy slices and the conformal limit statement will appear within the main body of the article. 

The Slodowy slices at fixed uniformizing points that we introduce here are intrinsically Lie theoretic constructions built from the special $\mathfrak{sl}_2$-triples associated to the $\Theta$-positive structures. Thus, in a sense, these slices generalize the Hitchin section constructed using a principal $\mathfrak{sl}_2$-triple. More abstract Slodowy slices for $\lambda$-connections where considered in \cite{CS}, and  families of 
$G$-opers were parameterized for arbitrary parabolic subgroups $P \subset G$ associated to even nilpotent elements of the Lie algebra of $G$. In particular, the relationship between
the Slodowy functor and higher Teichm\"{u}ller spaces was remarked in \cite[Section 6.4]{CS}. Furthermore, in \cite{CW}, the existence of the conformal limit was established for each stratum of the Bia\l{}ynicki-Birula decomposition of the space of $\lambda$-connections for a Higgs bundle with stable $\mathbb{C}^*$-limit. These articles have provided important insights for the present work.

\vspace{2mm}

\noindent\textbf{Motivation.} Our motivation was to build a mathematical framework for identifying precise families of solutions to the $G^{\mathbb{R}}$-Hitchin equations that produce $G^{\mathbb{R}}$-opers in their conformal limits. The families of solutions and of opers realized in this article lie in certain Slodowy slices which can be explicitly described and are parameterized by the same base space. It would be interesting to return to the original TBA-equations of \cite{Gaiotto} and \cite{GMN} describing the metric on the moduli space of four-dimensional $\mathcal{N}=2$ gauge theories compactified on a circle and interpret our Slodowy slices as solution sets having ``good'' conformal limits. Gaiotto's original prediction in \cite{Gaiotto} was that the resulting submanifold at the scaling limit is associated to the boundary condition defined by the compactification of the four-dimensional theory on a deformed cigar. The Slodowy slices constructed here are considered for general simple Lie groups $G$ admitting a $\Theta$-positive structure. One is thus lead to consider whether this structure is related to an additional symmetry in the equations and their solutions as was the case in \cite{Gaiotto} and \cite{Opers16} for the case of the Hitchin section.

\bigskip
\noindent\textbf{Acknowledgements.}
We warmly thank Brian Collier, Henry Liu, Andrew Neitzke and Valdo Tatitscheff for many fruitful discussions, remarks and shared insights. G.K. is grateful to the Kavli Institute for the Physics and Mathematics of the Universe for its hospitality and to the Scientific Committee of the University of Patras through the program ``Medicos'' for support. We are also grateful to the Alexander von Humboldt Foundation and the grant PosLieRep - ERC 101018839 for sponsorship during our stay at the University of Heidelberg where this work was initiated.

\section{Higgs bundles in Cayley components and \texorpdfstring{$\Theta$}{Theta}-positive structures}

In this section we collect the necessary background for the development of our results, namely, twisted real Higgs bundles and their deformation theory, magical $\mathfrak{sl}_2$-triples and $\Theta$-positive triples, Cayley components inside the moduli space of real Higgs bundles. We study closely the $\mathbb{C}^*$-fixed points inside the Cayley components and harmonic reductions on the uniformizing Higgs bundles that lie in these components.  

\subsection{\texorpdfstring{$L$-twisted $G^{\mathbb{R}}$}{L-twisted GR}-Higgs bundles and infinitesimal deformations}

For a real reductive Lie group $G^{\mathbb{R}}$, let $H^\R \subset G^{\mathbb{R}}$ be a maximal compact subgroup and let $\mathfrak{g}^{\mathbb{R}}=\mathfrak{h}^\R\oplus \mathfrak{m}^\R$ be a Cartan decomposition of the Lie algebra $\mathfrak{g}^{\mathbb{R}}$ of $G^{\mathbb{R}}$, so that the Lie algebra structure of $\mathfrak{g}^{\mathbb{R}}$ satisfies the relations:
\[\left[ \mathfrak{h}^\R,\mathfrak{h}^\R \right]\subset \mathfrak{h}^\R,\,\,\,\,\,\,\,\left[ \mathfrak{h}^\R,\mathfrak{m}^\R \right]\subset \mathfrak{m}^\R,\,\,\,\,\,\,\,\left[ \mathfrak{m}^\R,\mathfrak{m}^\R \right]\subset \mathfrak{h}^\R.\]
We will drop the superscript $\R$ to denote the complexification of an object. Consider the complexification ${{\mathfrak{g}}}={{\mathfrak{h}}}\oplus {{\mathfrak{m}}}$  of this Cartan decomposition. The action of the group $H^\R$ on the Lie algebra $\mathfrak{m}^\R$ via the adjoint representation extends to a linear holomorphic action of ${{H}}$ on ${{\mathfrak{m}}}=\mathfrak{m}^\R\otimes_\R \mathbb{C}$:
\[\iota :{{H}}\to \text{Aut}\left( {{\mathfrak{m}}} \right).\]
The Killing form on $\mathfrak{g}^{\mathbb{R}}$ induces on $\mathfrak{m}$ a Hermitian structure preserved by the action on $H^\R$. 

Let now $X$ be a Riemann surface of genus $g \ge 2$ and let $K=T^*X$ be the canonical line bundle of $X$. For a holomorphic principal $H$-bundle $\calP_H$ over $X$, denote by $\calP_H(\mathfrak{m})=\calP_H \times_{H} \mathfrak{m}$, the $\mathfrak{m}$-bundle associated to $\calP_H$ via the isotropy representation $\iota$ above. For a fixed holomorphic line bundle $L$ over the Riemann surface $X$, we have the following:

\begin{definition}
  An \emph{$L$-twisted $G^{\mathbb{R}}$-Higgs bundle} over $X$ is a pair $(\calP, \Phi)$, where $\calP$ is a holomorphic principal $H$-bundle over $X$ and $\Phi$ is a holomorphic global section of $\calP(\mathfrak{m}) \otimes L$. \\
  Two $L$-twisted $G^{\mathbb{R}}$-Higgs bundles $(\calP,\Phi)$ and $(\calP', \Phi')$ will be called \emph{isomorphic} when there exists a holomorphic gauge transformation $f:\calP \to \calP'$ such that $\Phi=f^* \Phi'$. \\
  When $L=K$, we will be just referring to the pairs above as \emph{$G^{\mathbb{R}}$-Higgs bundles}.
\end{definition}

Let $\tilde{G}^{\mathbb{R}} \subset G^{\mathbb{R}}$ be a reductive subgroup of $G^{\mathbb{R}}$ with maximal compact subgroup $\tilde{H}^{\mathbb{R}} \subset \tilde{G}^{\mathbb{R}}$ and Cartan decomposition $\tilde{\mathfrak{g}} = \tilde{\mathfrak{h}} \oplus \tilde{\mathfrak{m}}$. A reduction of structure group of an $L$-twisted $G^{\mathbb{R}}$-Higgs bundle is defined as follows.

\begin{definition}
Let $(\calP, \Phi)$ be an $L$-twisted $G^{\mathbb{R}}$-Higgs bundle over $X$. A \emph{reduction of structure group} of the pair $(\calP, \Phi)$ to an $L$-twisted $\tilde{G}^{\mathbb{R}}$-Higgs bundle $(\tilde{\calP}, \tilde{\Phi})$ is given by the following data:
\begin{enumerate}
\item A holomorphic reduction of the structure group of the principal $H$-bundle $\calP$ over $X$  to a principal $\tilde{H}$-bundle $\tilde{\calP} \hookrightarrow \calP$ over $X$, and 
\item A holomorphic section $\tilde{\Phi}$ of $\tilde{\calP}(\tilde{\mathfrak{m}}) \otimes L$ which maps to $\Phi$ under the embedding $\tilde{\calP}(\tilde{\mathfrak{m}}) \otimes L \hookrightarrow \calP(\mathfrak{m}) \otimes L$.
\end{enumerate}
\end{definition}

In order to define a moduli space of $L$-twisted $G^{\mathbb{R}}$-Higgs bundles, one needs to consider suitable notions of semistability, stability and polystability in the principal bundle setting. These notions are defined in terms of an antidominant character for a parabolic subgroup and a holomorphic reduction of the structure group of the bundle $\calP$. Since we will not be examining these stability conditions on certain $G^{\mathbb{R}}$-Higgs bundles throughout the article, we shall not recall the notions here; the reader is referred for a detailed account to \cite{GGM2009} or \cite{Aparicio} in the simpler case when ${{H}}$ is semisimple in particular.

The \emph{moduli space of polystable $L$-twisted $G^{\mathbb{R}}$-Higgs bundles} over $X$ is defined as the set of isomorphism classes of polystable $L$-twisted $G^{\mathbb{R}}$-Higgs bundles $\left( \calP, \Phi  \right)$ modulo the complex gauge group. We will denote this space by  $\mathsf{\mathcal{M}}_L\left( G^{\mathbb{R}} \right)$, and whenever $L=K$ simply by $\mathsf{\mathcal{M}}\left( G^{\mathbb{R}} \right)$.

A systematic infinitesimal study of the moduli functor for $L$-twisted $G^{\mathbb{R}}$-Higgs bundles was initiated by Biswas and Ramanan in \cite{BR}. The infinitesimal deformation space of a stable pair turns out to be a suitable hypercohomology which can be explicitly computed. We next review standard deformation theory for $L$-twisted $G^{\mathbb{R}}$-Higgs bundles and refer the reader to \cite{BR} or \cite[Section 3.3]{GGM2009} for more details. 

\begin{definition} Let $\left( \calP, \Phi  \right)$ be an $L$-twisted $G^{\mathbb{R}}$-Higgs bundle over $X$. The \emph{deformation complex} of $\left( \calP, \Phi  \right)$ is the following complex of sheaves
  \begin{equation}\label{def_cx}
    {{C}^{\bullet }}( \calP, \Phi  ):\calP( {{\mathfrak{h}}} )\xrightarrow{d\iota(\Phi)} \calP( {{\mathfrak{m}}})\otimes L.
  \end{equation}
\end{definition}

Similarly to the fact that the infinitesimal deformation space of a holomorphic vector bundle $E$ is isomorphic to $H^1(\text{End}(E))$, the space of infinitesimal deformations of an $L$-twisted $G^{\mathbb{R}}$-Higgs bundle $(\calP, \Phi)$ is naturally isomorphic to the group $\mathbb{H}^1({{C}^{\bullet }}\left( \calP, \Phi  \right))$.

By definition of hypercohomology, there is a natural long exact sequence:
\begin{align*}
  0 & \to \mathbb{H}^0({{C}^{\bullet }}\left( \calP, \Phi  \right)) \to H^0(\calP(\mathfrak{h})) \xrightarrow{d\iota(\Phi)} H^0(\calP(\mathfrak{m}) \otimes L)\\
    & \to \mathbb{H}^1({{C}^{\bullet }}\left( \calP, \Phi  \right)) \to H^1(\calP(\mathfrak{h})) \xrightarrow{d\iota(\Phi)} H^1(\calP(\mathfrak{m}) \otimes L) \to  \mathbb{H}^2({{C}^{\bullet }}\left( \calP, \Phi  \right)) \to 0.
\end{align*}
Therefore, one gets the following characterization for the infinitesimal automorphism space $\text{aut}(\calP, \Phi)$ of a pair $(\calP, \Phi)$
\[\text{aut}(\calP, \Phi) = \{s \in H^0(\calP(\mathfrak{h})) \, \vert \, d\iota(s)(\Phi) =0 \} \simeq \mathbb{H}^0({{C}^{\bullet }}\left( \calP, \Phi  \right)).\]

Considering the dual of the deformation complex (\ref{def_cx}),
\begin{equation}\label{dual_def_cx}
  {{C}^{\bullet }}\left( \calP, \Phi  \right)^* \otimes K:\calP( {{\mathfrak{m}}}) \otimes L^{-1}K \xrightarrow{d\iota(\Phi)^*\otimes \text{Id}_K} \calP({{\mathfrak{h}}}) \otimes K,
\end{equation}
the Killing form on $\mathfrak{g}$ identifies $\mathcal{P}(\mathfrak{g})^*$ with $\mathcal{P}(\mathfrak{g})$, and $d\iota(\Phi)^*$ with $-d\iota(\Phi)$; now
Serre duality in hypercohomology provides 
a natural isomorphism
\[\mathbb{H}^2({{C}^{\bullet }}\left( \calP, \Phi  \right)) \simeq \mathbb{H}^0({{C}^{\bullet }}\left( \calP, \Phi  \right)^*\otimes K)^*,\]
for an $L$-twisted $G$-Higgs bundle $(\calP, \Phi)$, where $G$ is a complex Lie group. 

The next proposition gives the zeroth and the second hypercohomology group for a stable $L$-twisted $G^{\mathbb{R}}$-Higgs bundle:

\begin{proposition}\cite[Propositions 3.15 and 3.17]{GGM2009}
  \label{prop:hypercohom-of-G-Higgs-bundles}
  Let $G^{\mathbb{R}}$ be a real Lie group with complexification $G$, and let $(\calP, \Phi)$ be an $L$-twisted $G^{\mathbb{R}}$-Higgs bundle which can be also viewed as a $G$-Higgs bundle.
  \begin{enumerate}
  \item There is an isomorphism of deformation complexes
    \[{C}_{G}^{\bullet }\left( \calP, \Phi  \right) \simeq {C}_{G^{\mathbb{R}}}^{\bullet }\left( \calP, \Phi  \right) \oplus {C}_{G^{\mathbb{R}}}^{\bullet }\left( \calP, \Phi  \right)^* \otimes K. \]
    Moreover, Serre duality provides the isomorphism
    \[\mathbb{H}^0({C}_{G}^{\bullet }\left( \calP, \Phi  \right)) \simeq \mathbb{H}^0({C}_{G^{\mathbb{R}}}^{\bullet }\left( \calP, \Phi  \right)) \oplus \mathbb{H}^2({C}_{G^{\mathbb{R}}}^{\bullet }\left( \calP, \Phi  \right))^*. \]
  \item Assuming that $G^{\mathbb{R}}$ is a real semisimple Lie group with complexification $G$, then for an $L$-twisted $G^{\mathbb{R}}$-Higgs bundle $(E, \Phi)$ which is stable when viewed as a $G$-Higgs bundle, one has
    \[\mathbb{H}^0({C}_{G^{\mathbb{R}}}^{\bullet }\left( \calP, \Phi  \right)) =0=\mathbb{H}^2({C}_{G^{\mathbb{R}}}^{\bullet }\left( \calP, \Phi  \right)). \]
  \end{enumerate}\end{proposition}
We close this short review with the following useful result:
\begin{proposition}
  Let $(\calP, \Phi)$ be a stable $L$-twisted $G^{\mathbb{R}}$-Higgs bundle. If $(\calP,\Phi)$ is simple, that is, if $\text{Aut}(\calP, \Phi)=\text{ker}(\iota) \cap Z(H)$, and it is stable as a $G$-Higgs bundle, then it represents a smooth point in the moduli space $\mathcal{M}_L(G^{\mathbb{R}})$.
\end{proposition}

\subsection{\texorpdfstring{$G^{\mathbb{R}}$}{GR}-Higgs bundles and Hitchin equations}
Let $\left( \calP, \Phi  \right)$ be a $G^{\mathbb{R}}$-Higgs bundle over a compact Riemann surface $X$. By a slight abuse of notation we shall denote the underlying smooth objects of $\calP$ and $\Phi $ by the same symbols. The Higgs field can be thus viewed as a $\left( 1,0 \right)$-form $\Phi \in {{\Omega }^{1,0}}\left( \calP\left( {{\mathfrak{m}}} \right) \right)$. Given a reduction $\calQ$ of structure group to $H^\R$ in the ${{H}}$-bundle $\calP$, we denote by ${{F}_{D}}$ the curvature of the unique (up to scalar) connection $D$ compatible with $\calQ$ and the holomorphic structure on $\calP$. Let ${{\rho }_{\calQ}}:{{\Omega }^{1,0}}\left( \calP\left( {{\mathfrak{g}}} \right) \right)\to {{\Omega }^{0,1}}\left( \calP\left( {{\mathfrak{g}}} \right) \right)$ be defined by the complex conjugation of ${{\mathfrak{g}}}$, given fiberwise by the reduction $\calQ$ combined with complex conjugation on complex 1-forms. The next theorem was proved in \cite{GGM2009} for an arbitrary reductive real Lie group $G^{\mathbb{R}}$.

\begin{theorem}\cite[Theorem 3.21]{GGM2009}\label{thm:Hitchin-Kobayashi} There exists a reduction $\calQ$ of the structure group of $\calP$ from ${{H}}$ to $H^\R$ satisfying the Hitchin equation
  \[{{F}_{D}}-\left[ \Phi ,{{\rho }_{\calQ}}\left( \Phi  \right) \right]=0\]
  if and only if $\left( \calP, \Phi  \right)$ is polystable.
\end{theorem}
From the point of view of moduli spaces, it is convenient to fix a smooth principal $H^\R$-bundle ${{\textbf{P}}_{H^\R}}$ with fixed topological class $d$ in ${{\pi }_{1}}\left( H^\R \right)$ and study the moduli space of solutions to Hitchin's equations for a pair $\left( A, \Phi  \right)$ consisting of an $H^\R$-connection 1-form $A$ and $\Phi \in {{\Omega }^{1,0}}\left( X,\textbf{P}_{H^\R}\left( {{\mathfrak{m}}} \right) \right)$:
\begin{align}
  {{F}_{A}}-\left[ \Phi ,\rho_\calQ \left( \Phi  \right) \right] &= 0 \label{Hit_eq:1} \\
  {{\bar{\partial }}_{A}}\Phi &= 0, \label{Hit_eq:2}
\end{align}
where ${{d}_{A}}$ is the covariant derivative associated to $A$, and ${{\bar{\partial }}_{A}}$ is the $\left( 0,1 \right)$-part of ${{d}_{A}}$ defining the holomorphic structure on ${{\textbf{P}}_{H^\R}}$. Moreover, $\rho_\calQ $ is defined by the fixed reduction of structure group ${\textbf{P}_{H^\R}}\hookrightarrow {\textbf{P}_{H^\R}}\left( {{H}} \right)$. The gauge group ${{\mathsf{\mathcal{G}}}_{H^\R}}$ of ${\textbf{P}_{H^\R}}$ acts on the space of solutions by conjugation and the moduli space of solutions is defined by
\[\mathsf{\mathcal{M}}^{\text{gauge}}\left( G^{\mathbb{R}} \right):={\left\{ \left( A, \Phi  \right)\text{ satisfying (\ref{Hit_eq:1}) and (\ref{Hit_eq:2})} \right\}}/{{{\mathsf{\mathcal{G}}}_{H^\R}}}\;\]
Now, Theorem \ref{thm:Hitchin-Kobayashi} implies the existence of a bijection
\[{\mathsf{\mathcal{M}}}\left( G^{\mathbb{R}} \right)\cong \mathsf{\mathcal{M}}^{\text{gauge}}\left( G^{\mathbb{R}} \right).\]

Using the one-to-one correspondence between $H^\R$-connections on ${\textbf{P}_{H^\R}}$ and $\bar{\partial }$-operators on ${\textbf{P}_{{{H}}}}$, the homeomorphism in the above theorem can be interpreted by saying that in the  $\mathsf{\mathcal{G}}_{H}$-orbit of a polystable $G^{\mathbb{R}}$-Higgs bundle we can find another Higgs bundle $\left( \calP, \Phi  \right)$ whose corresponding pair $\left( {{d}_{A}},\Phi  \right)$ satisfies the equation ${{F}_{A}}-\left[ \Phi ,\rho_\calQ \left( \Phi  \right) \right]=0$, and this is unique up to $H^\R$-gauge transformations.

\subsection{Harmonic reductions}
For our purposes in this article, we shall consider the following rescaled version of the system of nonlinear PDE for the data  $(\calP, \rho_{\calQ(R)}, \calQ(R), \Phi)$:
\begin{align}
  {{F}_{D_\calQ(R)}}-R^2\left[ \Phi ,\rho_{\calQ(R)} \left( \Phi  \right) \right] &= 0 \label{resc_Hit_eq:1}\\
  {{\bar{\partial }}_{D_{\calQ(R)}}}\Phi &= 0, \label{resc_Hit_eq:2}
\end{align}
obtained by replacing $\Phi$ with $R \Phi$, for a parameter $R \in \mathbb{R}^+$. We will be referring to this set of equations as \emph{Hitchin's equations with parameter $R$}.

Fix a $G^{\mathbb{R}}$-Higgs bundle $(\calP, \Phi)$ and a parameter $R$ as above. Since the involution $\rho: \mathfrak{g} \to \mathfrak{g}$ is $H^\R$-invariant, a reduction of structure group $\calQ(R)$ induces an involution $\rho_{\calQ(R)}$ of $\text{ad}\calP$ acting trivially on $\text{ad}\calQ(R)$. There is a unique connection $D_{\calQ(R)}$ in $\text{ad}\calQ(R)$ which decomposes into its $(0,1)$ and $(1,0)$-parts as 
\[D_{\calQ(R)} = \bar{\partial}_{\calP} + \partial^{\calQ(R)}_{\calP},\]
where $\partial^{\calQ(R)}_{\calP}:= \rho_{\calQ(R)} \circ \partial_{\calP} \circ \rho_{\calQ(R)}$. We denote the curvature of this connection by $F_{D_{\calQ(R)}} \in \Omega^2(X, \text{ad}\calQ(R))$. We shall call a reduction of structure group $\calQ(R)$ of $\calP$ from $H$ to $H^\R$ such that 
\[F_{D_{\calQ(R)}} -R^2 \left[ \Phi ,\rho_{\calQ(R)} ( \Phi ) \right] =0, \]
a \emph{harmonic reduction} with parameter $R$. 

Given a $G^{\mathbb{R}}$-Higgs bundle  that satisfies Hitchin's equations with parameter $R \in \mathbb{R}^+$, there is a corresponding family of non-unitary connections on $E$ given by the family  
\[\nabla_{\zeta} = \zeta^{-1}R\Phi +D_{\calQ(R)} - \zeta R \rho_{\calQ(R)}(\Phi), \,\, \zeta \in \mathbb{C}^*.\]
This family is flat for all $\zeta \in \mathbb{C}^*$.

\subsection{Magical \texorpdfstring{$\fraksl_2$}{sl2}-triples}
In this subsection, we collect the necessary Lie theoretic background that we shall need in the sequel. The main references for this part are \cite{Cayley} and \cite{CoMc}. 

Let $\frakg$ be a complex simple Lie algebra and $0 \neq e \in \frakg$ be a nilpotent element. By the Jacobson--Morozov theorem, there is a unique $\fraksl_2$-triple $\set{f,h,e}$ that contains $e$ and satisfies the relations
\begin{equation*}
  [h,f]=-2f, \quad [h,e]=2e, \quad [e,f]=h. 
\end{equation*}
Let $\fraks=\ang{f,h,e}$ be the Lie subalgebra generated by the $\fraksl_2$-triple. The Lie algebra $\mathfrak{g}$ admits a decomposition as an $\fraks$-module,
\begin{equation*}
  \frakg = W(0) \oplus \bigoplus_{j=1}^M W(m_j), 
\end{equation*}
where $W(k)$ is isomorphic to a direct sum of $n_k \neq 0$ copies of the $(k+1)$-dimensional irreducible representation of $\fraks$ of highest weight $k$.

On the other hand, for $h \in \frakg$ any semisimple element,  $\ad_h$ induces a characteristic grading on $\frakg$ 
\begin{align*}
  \frakg = \bigoplus_{k \in \Z} \frakg_k, 
\end{align*}
where $X \in \frakg_k$, if $\ad_h(X)=kX$. Moreover, this grading respects the adjoint action of the Lie algebra:
\begin{align*}
  [\frakg_k, \frakg_l] \subset \frakg_{k+l}. 
\end{align*}

If $\set{f,h,e}$ is an $\fraksl_2$-triple, then $\ad_f, \ad_e$ are graded linear maps of degrees $-2$ and $2$ respectively.

Let now $W(k)_l = W(k) \cap \frakg_l$ denote the $l$-eigenspace of $\ad_h$ in the weight module $W(k)$. Let $\sigma: \frakg \to \frakg$ be the complex linear involution given by
\begin{equation*}
  \sigma(x) =
  \begin{cases}
    x, & \; \text{if} \; x \in W(0) \\
    (-1)^{i+1}x, & \; \text{if} \; x \in W(k)_{k-2i}, \; \text{ for all } k \in \set{2m_j}_{j=1}^M, \quad 0 \le i \le k. 
  \end{cases}
\end{equation*}
Following \cite{Cayley} we set the next definition:
\begin{definition}\label{def:magical}
  For a nilpotent element $e \in \mathfrak{g}$ of a complex simple Lie algebra $\mathfrak{g}$, an $\mathfrak{sl}_2$-triple $\{f,h,e\}$ and an involution $\sigma:\mathfrak{g} \to \mathfrak{g}$ as above, the triple $\{f,h,e\}$ and the element $e$ will be called \emph{magical}, if $\sigma$ is a Lie algebra homomorphism. We also refer to the involution $\sigma$ as \emph{magical} if it is given by a magical triple $\set{f,h,e}$.
\end{definition}

\begin{remark} Note that if $\set{f,h,e}$ is a magical $\fraksl_2$-triple, then all the eigenvalues of $\ad_h$ are even due to the classification theorem of magical $\fraksl_2$-triples \cite[Theorem 3.1]{Cayley}. In particular, all the highest weights in the decomposition of $\frakg$ are even, hence from now on we will write $k=2m_j$ for the highest weights. In general, an $\fraksl_2$-triple that has this property is called \textit{even}. 
\end{remark}

\subsection{Real forms}

Fix a real form $\tau$ on $\frakg$ such that $\rho:=\tau\sigma=\sigma\tau$ is a compact real form on $\frakg$ and such that $\tau$ restricts to the magical triple as follows: 
\begin{align*}
  \tau(h)=-h, \quad \tau(e)=f, \quad \tau(f)=e. 
\end{align*}

\begin{definition} Let $\frakg^\R := \frakg^\tau$ denote the fixed point set of $\tau$ as above. We will refer to both $\tau$ and $\frakg^\R$ as the \emph{canonical real form} associated to the magical triple.
\end{definition}

Since $\rho$ is a compact real form, the $\pm 1$-eigenspace decomposition of $\sigma|_{\frakg^\R}=\rho|_{\frakg^\R}$ gives a Cartan decomposition on $\frakg^\R$:
\begin{align*}
  \frakg^\R = \frakh^\R \oplus \frakm^\R. 
\end{align*}
It then makes sense to refer to the $\pm 1$-eigenspace decomposition of $\sigma$ on $\frakg$ as the \emph{complexified Cartan decomposition} of the Cartan involution $\sigma$ for the real form $\tau$:
\begin{align*}
  \frakg = \frakh \oplus \frakm. 
\end{align*}

By definition of the magical involution $\sigma$, we have further decompositions
\begin{align*}
  \frakh = \frakc \oplus \bigoplus_{j=1}^{M} \bigoplus_{i=1}^{m_j} \ad_f^{2i-1}(V_{2m_j}), \quad \frakm = \bigoplus_{j=1}^{M} \bigoplus_{i=0}^{m_j} \ad_f^{2i}(V_{2m_j}). 
\end{align*}


\subsection{Slodowy slices}

Let $V:=\ker(\ad_e)$ denote the centralizer of $e$ in $\frakg$. It decomposes as a sum of highest weight spaces
\begin{equation*}
  V = W(0) \oplus \bigoplus_{j=1}^M W(2m_j)_{2m_j}. 
\end{equation*}
For simplicity we write $V_k=W(k)_k$ to mean the highest weight spaces for each of the weights $k=2m_j$. For $k=0$, since $V_0$ is the Lie subalgebra that centralizes $\fraks$, we also denote it by
\begin{align*}
  \frakc:=V_0=W(0)=Z_{\frakg}(\fraks). 
\end{align*}
Note that $\frakc$ is a reductive Lie subalgebra, because the centralizer of a reductive group is reductive. The next proposition further characterizes the centralizer $\mathfrak{c}$.

\begin{proposition}\cite[Proposition 2.2]{Cayley}
  The centralizer $V$ of $e$ is a subalgebra in $\frakg$. Moreover, $\frakc$ is a reductive subalgebra and $\bigoplus_{j=1}^M V_{2m_j}$ is a nilpotent subalgebra. 
\end{proposition}

The affine space $f+V$ is a slice of the adjoint action of $G$ on $\frakg$ through $f$, called a \emph{Slodowy slice}. The centralizer $\frakc$ preserves the Slodowy slice.

\subsection{Cayley real forms}

Similar to the highest weight spaces, we pay special attention to the zero weight spaces: let $Z_k=W(k)_0$ for the highest weights $k=2m_j$ and $Z_0=\frakc$. In particular, the zero eigenspace of $\ad_h$ can be decomposed as
\begin{equation*}
  \frakg_0 = \frakc \oplus \bigoplus_{j=1}^M Z_{2m_j}. 
\end{equation*}
For the zero eigenspace $\mathfrak{g}_0$, one has the next proposition.
\begin{proposition}\cite[Proposition 2.7]{Cayley}
  Let $\theta: \frakg_0 \to \frakg_0$ be the vector space involution given by 
  \begin{align*}
    \theta(x)= \begin{cases}
      x, & \; \text{if} \; x \in \frakc \\
      -x, & \; \text{otherwise}. 
    \end{cases}
  \end{align*}
  Then $\theta$ is a Lie algebra involution on $\frakg_0$, which will be called the \emph{Cayley involution} associated to the magical triple $\set{f,h,e}$. 
\end{proposition}
\begin{remark}
  Note that for the involution $\theta$ above, it holds that  $\theta \neq \sigma|_{\frakg_0}$, because $\theta(h)=-h$ and $\sigma(h)=h$.
\end{remark}

\begin{lemma}
  Let $\mathfrak{g}$ be a complex simple Lie algebra and fix a magical triple $\set{f,h,e}$. Let $\rho=\tau\sigma=\sigma\tau$ be the compact real form on $\frakg$ associated to the magical involution $\sigma$ and the canonical real form $\tau$ associated to the magical triple. Then $\rho$ preserves the zero eigenspace $\frakg_0$ of $\ad_h$. 
\end{lemma}

\begin{proof}
  By definition of the magical involution, $\sigma$ acts on $\frakg_0$ by $\pm 1$, so $\sigma$ preserves $\frakg_0$. Moreover, by assumption the canonical real form $\tau$ takes $h$ to $-h$, so the zero eigenspace is preserved. Hence $\rho(\frakg_0) = \frakg_0$. 
\end{proof}

We may restrict $\rho$ to $\frakg_0$ to obtain a compact real form on $\frakg_0$. Now, let
\begin{align*}
  \tau_\calC := \theta \circ \rho|_{\frakg_0}. 
\end{align*}
We then call $\tau_\calC$ and its fixed point set $\frakg_\calC^\R := \frakg_0^{\tau_\calC}$, the \emph{Cayley real form} associated to the magical triple $\set{f,h,e}$.

\subsection{Magical triples as principal triples}

Fix a magical triple $\fraks = \ang{f,h,e}$ in a complex simple Lie algebra $\mathfrak{g}$, and let $\frakc=Z_{\frakg}(\fraks)$ be its centralizer. Let $\frakg(e)=Z_{\frakc^\perp}(\frakc)$ be the double centralizer of $\fraks$ complementary to the centralizer $\frakc$, that is, 
\begin{align*}
  Z_{\frakg}(Z_{\frakg}(\fraks)) = Z_\frakg(\frakc) = Z_{\frakc}(\frakc) \oplus \frakg(e). 
\end{align*}
A nilpotent element $e \in \mathfrak{g}$ is called \emph{principal} if its centralizer has dimension $\text{rk}(\mathfrak{g})$; then an $\mathfrak{sl}_2$-triple where such $e$ embeds into is called thereupon \emph{principal}. 
\begin{lemma}
  Let $\fraks$ be a magical $\fraksl_2$-triple in $\frakg$ with $\frakc = Z_\frakg(\fraks) = 0$, then $\fraks$ is principal. 
\end{lemma}

\begin{proof}
  A magical $\fraksl_2$ is even, thus every highest weight vector corresponds to a zero eigenvector of $\ad_h$. It is sufficient to show that $\dim \ker(\ad_e) = \dim \frakg_0 = \rk(\frakg)$. Using the magical involution from Definition \ref{def:magical}, it can be shown that when $\frakc = 0$, the subalgebra $\frakg_0$ is abelian (\cite[Proposition 2.7]{Cayley}. Hence $\frakg_0$ must be a maximal abelian subalgebra containing $h$, which has dimension $\rk(\frakg)$. Alternatively, the statement can be shown due to the classification of magical $\fraksl_2$-triples \cite[Theorem 3.1]{Cayley}. 
\end{proof}

It is clear that $\fraks$ centralizes $\frakc$ and belongs to the complement $\frakg(e)$. Moreover, $\fraks$ remains a magical $\fraksl_2$ in the reductive subalgebra $\frakg(e)$. Then $Z_{\frakg(e)}(\fraks)=\frakc \cap \frakg(e)=0$, and one can show that the magical triple is principal in $\frakg(e)$ (\cite[Proposition 4.5]{Cayley}).

The following characterization of the Cayley real forms will be also useful:

\begin{proposition}\cite[Propositions 4.2 and 4.8]{Cayley}
  Let $\set{f,h,e}$ be a magical triple. The zero eigenspace of $\ad_h$ admits a decomposition
  \begin{align*}
    \frakg_0 = \frakg_0^s \oplus (\frakg_0 \cap \frakg(e)) = \frakg_0^s \oplus \C^{r(e)}, 
  \end{align*}
  where $\frakg_0^s$ is a simple Lie algebra, and $r(e) := \rk(\frakg(e))$.
  In particular, the Cayley real form inherits the decomposition into simple and abelian subalgebras as
  \begin{align*}
    \frakg_\calC^\R = \frakg_\calC^{\R,s} \oplus \R^{r(e)}. 
  \end{align*}
\end{proposition}

\begin{lemma}
  Let $\frakc^\R = \frakc \cap \frakg^\R$ be the real form in $\frakc$ given by the canonical real form. Then the simple part of the Cayley real form admits a Cartan decomposition
  \begin{align*}
    \frakg_\calC^{\R,s} = \frakc^\R \oplus \frakm_\calC^{\R}, 
  \end{align*}
  where $\frakc^\R$ is the maximal compact subalgebra, and $\frakm_\calC^{\R}$ is the complement subspace. 
\end{lemma}

\begin{proof}
  By construction the Cayley involution $\theta$ is the Cartan involution associated to the Cayley real form $\tau_\calC$, and $\theta$ has $\frakc$ as fixed point set. Moreover, the Cayley real form is equivalent to the canonical real form on $\frakc$ given by
  \begin{align*}
    \tau_\calC|_{\frakc} = (\theta \circ \rho)|_{\frakc} = \rho|_{\frakc} = (\tau \circ \sigma)|_{\frakc} = \tau|_{\frakc}. 
  \end{align*}
  Hence the fixed point set of $\theta|_{\frakg_{\calC^\R}}$ in the Cayley real form is $\frakc^\R$. 
\end{proof}

The next proposition describes the weight space decomposition of $\mathfrak{g}$ with respect to a magical triple.

\begin{proposition}\cite[Lemma 5.7]{Cayley}
  Let $\set{f,h,e}$ be a magical triple in $\frakg$. Then $\frakg$ admits a decomposition with respect to the magical triple of the form
  \begin{align*}
    \frakg = W(0) \oplus W(2m_{\calC})\setminus \frakg(e) \oplus \bigoplus_{i=1}^{r(e)} W(2l_i) \cap \frakg(e), 
  \end{align*}
  where $W_{2m_{\calC}}$ is the unique non-trivial weight module which intersects $\frakg_\calC^{\R,s}$ non-trivially, and $l_i$ are the exponents of $\frakg(e)$. 
\end{proposition}

\begin{remark}
  The weight module $W(2m_{\calC})$ need not be fully contained in $\frakg_\calC^{\R,s}$. This occurs precisely when $W(2m_{\calC})$ intersects $\frakg(e)$ non-trivially, or in other words $m_{\calC} = l_i$ for some exponent $l_i$ of $\frakg(e)$. For example, in Case (2) of the classification of magical triples \cite[Theorem 3.1]{Cayley}, where the canonical real form $G^\R$ is a Hermitian Lie group of tube type, there is only one non-zero weight $2m_{\calC}=2l_1=2$, where $\frakg(e) = \fraks$ is the magical $\fraksl_2$ subalgebra. In this case, we write the decomposition as
  \begin{align*}
    \frakg = W(0) \oplus W(2) \setminus \fraks \oplus \fraks. 
  \end{align*}
\end{remark}

Note that since $W(2m_{\calC})$ is the only nontrivial weight module which intersects $\frakg_\calC^{\R,s}$, we have
\begin{equation}
  \label{eq:ZmC-frakmC}
  Z_{2m_{\calC}} = \frakm_\calC \oplus Z_{2m_{\calC}} \cap \frakg(e). 
\end{equation}

\subsection{Cayley components}
In \cite{Cayley}, the Cayley real form was used in order to provide a concrete description of certain connected components in the moduli space $\mathcal{M}(G^{\mathbb{R}})$ associated to these real forms. The connected components are parameterized by a map which can be explicitly described as follows. 
\begin{theorem}\cite[Theorem B]{Cayley}
  Let $G_\calC^{\R,s}$ and $G^\R$ be subgroups of $G$ with Lie algebras $\frakg_\calC^{\R,s}$ and $\frakg^\R$ respectively. Let $C$ and $H$ denote the complexification of the maximal compact subgroups of $G_\calC^{\R,s}$ and $G^\R$ respectively. 
  The Cayley map defined by
  \begin{align}
    \label{eq:cayley-map}
    &\Psi: \calM_{K^{m_{\calC}+1}}(G_\calC^{\R,s}) \times \bigoplus_{i=1}^{r(e)} H^0(K^{l_i+1}) \to \calM(G^\R), \quad \text{with} \\
    &\Psi((\calP_C, \psi), q_1,\ldots,q_{r(e)}) := (\calP_H, f + \tilde{\psi} + \sum_{i=1}^{r(e)} \tilde{q}_i), 
  \end{align}
  where $\tilde{\psi} = \ad_f^{-m_{\calC}}(\psi), \tilde{q}_i = \ad_f^{-m_i}(q_i)$, is an open and closed injection on moduli spaces.
\end{theorem}

The image of $\Psi$ is a union of connected components in the moduli space $\calM(G^\R)$, which will be called the \emph{Cayley components}.

\subsection{Theta-positive structures}
\label{sec:theta_pos_str}
The notion of a $\Theta$-positive structure for a simple real Lie group is a generalization of Lusztig's total positivity condition in \cite{Lusztig} and was introduced by Guichard and Wienhard in \cite{GW} in order to distinguish connected components of $\Theta$-positive surface group representations. This notion is given in terms of properties of the Lie algebra of parabolic subgroups of the simple real Lie group defined by a subset of simple positive roots. We briefly review this structure next and refer to \cite{GW} for complete descriptions.  

Let $G^\R$ be a connected simple real Lie group with finite centre. Let $\Sigma^\R$ be the restricted root system associated to $G^\R$, and let $\Delta^\R \subset \Sigma^{\R,+}$ denote the set of positive simple roots. Fix a subset $\Theta \subset \Delta^\R$ of positive simple roots.
Let $\Sigma^{\R,+}_\Theta = \Sigma^{\R,+} \setminus \ang{\Delta^\R \setminus \Theta}$. There are two nilpotent subalgebras associated to $\Theta$: 
\begin{align*}
  \fraku_\Theta^{\R,\pm} = \sum_{\alpha \in \Sigma^{\R,+}_\Theta} \frakg^\R_{\pm \alpha}. 
\end{align*}
The normalizers of the nilpotent subalgebras $P_\Theta^{\R,\pm} := N_{G^\R}(\fraku_\Theta^{\R,\pm})$ are parabolic subgroups of $G^\R$. Let $L_\Theta^\R = P_\Theta^{\R,+} \cap P_\Theta^{\R,-}$ be the Levi subgroup of the parabolic subgroups, and write $L_\Theta^{\R,0}$ for the identity component in $L_\Theta^\R$.

For each $\beta \in \Sigma^\R$, let now
\begin{align*}
  \fraku_\beta^{\R,\pm} = \sum_{\alpha \in \Sigma^\R, \, \alpha - \beta \in \ang{\Delta^\R\setminus \Theta}} \frakg^\R_{\pm \alpha}. 
\end{align*}

A \emph{$\Theta$-positive structure} on $G^\R$ is a set of non-trivial convex cones $c_\alpha \subset \fraku_\alpha^\R$ labelled by $\alpha \in \Theta$, which are acute and $L_\Theta^{\R,0}$-invariant. 

\begin{theorem}\cite[Theorem 3.4]{GW}
  The simple real groups $G^\R$ that admit a $\Theta$-positive structure are classified by the following four families: 
  \begin{enumerate}
  \item Split real groups.
  \item Hermitian groups of tube type.
  \item Groups locally isomorphic to $\SO(p+1,p+k)$, $p, k > 1$. 
  \item The real forms of $E_6, E_7, E_8$ or $F_4$, whose reduced root system is of type $F_4$. 
  \end{enumerate}
  In the first case above, a split real group admits a $\Theta$-positive structure for $\Theta = \Delta^\R$ the entire restricted simple roots; in all other cases $\Theta$ consists of the set of long restricted simple roots. 
\end{theorem}

Let $W$ be the Weyl group associated to the root system $\Delta^\R$ of $\frakg^\R$. For a subset $F \subset \Delta^\R$ of simple roots, we write $W_F = \ang{s_\alpha \in W \mid \alpha \in F}$ to mean the subgroup of the Weyl group generated by reflections associated to the roots in $F$.

For each root $\alpha \in \Theta$, there exists a $W_{\Delta^\R \setminus \Theta}$-invariant choice of $\fraksl_2\R$-triple $\set{f_\alpha, h_\alpha, e_\alpha}$ satisfying (see \cite[Section 3.3]{GW}):
\begin{itemize}
\item $e_\alpha \in c_\alpha^\circ$, where $c_\alpha^\circ$ is the interior of the convex cone from the $\Theta$-positive structure. 
\item $f_\alpha = -\rho(e_\alpha)$, for $\rho$ the Cartan involution on $\frakg^\R$. 
\item $h_\alpha = [e_\alpha, -\rho(f_\alpha)]$. 
\end{itemize}
Note that $f_\alpha \in c_\alpha^{\circ,-}$, the opposite cone to $c_\alpha^{\circ}$. The set $\set{f_\alpha, h_\alpha, e_\alpha}_{\alpha \in \Theta}$ is called a \emph{$\Theta$-base} of $\frakg^\R$.

Let $w_F$ denote the longest word in $W_F$. For $\alpha_\Theta$, the unique root in $\Theta$ which is connected to $\Delta^\R \setminus \Theta$ in the Dynkin diagram, let $\sigma_{\alpha_\Theta} \in W$ be defined by the equation 
\begin{align*}
  w_{\set{\alpha_\Theta} \cup \Delta^\R \setminus \Theta} = \sigma_{\alpha_\Theta} w_{\Delta^\R \setminus \Theta}. 
\end{align*}
Then, the \emph{$\Theta$-Weyl group} is the subgroup $W(\Theta) < W$ generated by the reflections
\begin{align*}
  R(\Theta) = \set{s_\alpha, \sigma_{\alpha_\Theta} \in W \mid \alpha \in \Theta \setminus \set{\alpha_\Theta}}. 
\end{align*}

\begin{proposition}\cite[Proposition 4.7]{GW}
  Let $G^\R$ be a simple real Lie group with a $\Theta$-positive structure. Then, the Coxeter system $(W(\Theta), R(\Theta))$ can be of one of the following Lie types: 
  \begin{enumerate}
  \item For a split real group, the same type as the Coxeter system associated to ($\frakg^\R,\Delta^\R$). 
  \item For a Hermitian group of tube type, the type $A_1$.
  \item For a group locally isomorphic to $\SO(p+1,p+k)$, $p, k > 1$, the type $B_p$.
  \item For a real form of $E_6, E_7, E_8$ or $F_4$ with reduced root system of type $F_4$, the type $G_2$. 
  \end{enumerate}
\end{proposition}

\begin{theorem}\cite[Theorem 5.1]{GW}
  Let $\frakg_\Theta^\R = \ang{f_\alpha, h_\alpha, e_\alpha}_{\alpha \in \Theta}$ be the subalgebra generated by the $\Theta$-base of $\frakg^\R$. Then, $\frakg_\Theta^\R$ is a split real Lie algebra of Lie type $(W(\Theta),R(\Theta))$. 
\end{theorem}

Let $\Theta^+$ denote the set of positive roots generated by $\Theta$. For each $\alpha \in \Theta^+ \setminus \Theta$, fix an $\fraksl_2\R$-triple $\set{f_\alpha, h_\alpha, e_\alpha}$ in $\frakg_\Theta^\R$. Set
\begin{align*}
  h_\Theta := \sum_{\alpha \in \Theta^+} h_\alpha = \sum_{\alpha \in \Theta} r_\alpha h_\alpha, r_\alpha \in \Z^+.
\end{align*}
The second equality comes from the fact that $\set{h_\alpha}_{\alpha \in \Theta}$ is a basis of the Cartan subalgebra generated by $\Theta$. Then set
\begin{align*}
  e_\Theta := \sum_{\alpha \in \Theta} \sqrt{r_\alpha} e_\alpha, \quad f_\Theta := \sum_{\alpha \in \Theta} \sqrt{r_\alpha} f_\alpha. 
\end{align*}

By construction, $\set{f_\Theta,h_\Theta,e_\Theta}$ is a principal $\fraksl_2\R$-triple in $\frakg_\Theta^\R$. We will call it the \emph{$\Theta$-principal $\fraksl_2\R$-triple}.

\subsection{Connection between \texorpdfstring{$\Theta$}{Theta}-principal and magical triples}

Up to appropriate notions of equivalence relations, $\Theta$-principal $\fraksl_2\R$-triples are in bijection with magical $\fraksl_2\C$-triples via a certain Cayley transform. A Cayley transform is a conjugation by $\SL(2,\C)$ that interchanges the subalgebras $\fraksl_2\R$ and $\mathfrak{su}_{1,1}$ inside $\fraksl_2\C$. Fix the complexified Cartan involution $\sigma(X)=-X^t$ for the real form $\fraksl_2\R$. There are exactly two distinct $\SL(2,\R)$-orbits of $\sigma$-stable Cartan subalgebras. For each of these, we can identify a standard triple as described next.

The triple
\begin{equation*}
  f =
  \begin{pmatrix}
    0 & 0 \\
    1 & 0
  \end{pmatrix}, \quad
  h =
  \begin{pmatrix}
    1 & 0 \\
    0 & -1
  \end{pmatrix}, \quad
  e =
  \begin{pmatrix}
    0 & 1 \\
    0 & 0
  \end{pmatrix}
\end{equation*}
corresponds to the maximal noncompact Cartan subalgebra $\R h$ (see \cite[Chapter VI]{Knapp}). Moreover, a second triple that corresponds to the maximal compact Cartan subalgebra $\R h'$ is given by
\begin{equation*}
  f' = \frac{1}{2}
  \begin{pmatrix}
    i & -1 \\
    -1 & -i
  \end{pmatrix}, \quad
  h' =
  \begin{pmatrix}
    0 & i \\
    -i & 0
  \end{pmatrix}, \quad
  e' = \frac{1}{2}
  \begin{pmatrix}
    i & 1 \\
    1 & -i
  \end{pmatrix}. 
\end{equation*}
Note that in \cite{CoMc}, the first triple $\{f,h,e\}$ as above is referred to as a \emph{Cayley triple}, which is defined by the property
\begin{align*}
  \sigma(f)=-e, \quad \sigma(h)=-h, \quad \sigma(e)=-f; 
\end{align*}
while the second one $\{f',h',e'\}$ is called a \emph{normal triple}, which is defined by the property
\begin{align*}
  \sigma(f')=-f', \quad \sigma(h')=h', \quad \sigma(e')=-e'. 
\end{align*}

There is an element in the complex Lie group $\SL(2,\C)$ that conjugates these Cartan subalgebras. Indeed, consider
\begin{equation*}
  g = \exp(-\frac{\pi}{4} i (e+f)) = \frac{1}{\sqrt{2}}
  \begin{pmatrix}
    1 & -i \\
    -i & 1
  \end{pmatrix}. 
\end{equation*}
Applying $\Ad(g)\cdot X = g^{-1} X g$ then gives
\begin{equation*}
  f' = \frac{1}{2}(e+f-ih), \quad h' = i(e-f), \quad e' = \frac{1}{2}(e+f+ih). 
\end{equation*}
We call this $\Ad(g)$ as above and its inverse \emph{Cayley transforms}. 

Let now $\pi_\Theta^\C: \fraksl_2\C \to \frakg := \frakg^\R \otimes_\R \C$ be the complexification of a $\Theta$-principal embedding. Let $G$ be the complexification of $G^\R$ and by a small abuse of notation let $g \in G$ denote the image of $g$ from above induced by the complexified $\Theta$-principal embedding. One now has the following:

\begin{theorem}\cite[Theorem 8.14]{Cayley}
  The triple $(f,h,e)$ is a $\Theta$-principal triple in $\mathfrak{g}^\R$ if and only if
  \begin{align*}
    (f',h',e') = \Ad(g)(f,h,e)
  \end{align*}
  is a magical $\fraksl_2\C$-triple in $\frakg$. 
\end{theorem}

\begin{remark} \hfill
  \begin{enumerate} 
  \item Note that given a magical triple, there is a canonical real form $\frakg^\R$ associated to it. 
  \item If we weaken the notion of equivalence to include conjugation by the entire complex Lie algebra $\frakg$, the $\Theta$-principal triple and the magical $\fraksl_2\C$-triple above are conjugate by the Cayley transform. Their adjoint representations on $\frakg$ are equivalent up to conjugation by $\frakg$. In particular, if $\sigma$ is a magical involution for a magical triple $(f,h,e)$, then $g^{-1}\circ \sigma \circ g$ is a magical involution for the corresponding $\Theta$-positive triple. Hence a $\Theta$-positive triple in $\frakg^\R$ is magical in $\frakg$, the complexification of the real Lie algebra $\frakg^\R$. 
  \end{enumerate}
\end{remark}

\section{Cayley uniformizing Higgs bundles}\label{sec:Cayley_uniform_Higgs}

Suppose $H_1, H_2 < G$ are two commuting subgroups. For $i=1,2$, let $\calP_{H_i}$ be principal $H_i$-bundles on $X$. Then we write
\begin{align*}
  \calP_{H_1} \star \calP_{H_2}(G) := (\calP_{H_1} \times_X \calP_{H_2}) \times_{H_1 \times H_2} G
\end{align*}
to denote the product $G$-bundle. Note that since the subgroups commute in $G$, we have
\begin{align*}
  \calP_{H_1} \star \calP_{H_2}(G) = \calP_{H_2} \star \calP_{H_1}(G). 
\end{align*}

\subsection{\texorpdfstring{$\C^*$}{C*}-fixed points}
\label{subsection:cstar-fixed-points}
A natural action of the group $\mathbb{C}^*$ exists on the moduli space $\mathcal{M}_L(G^{\mathbb{R}})$ given by $(\calP, \Phi) \mapsto (\calP, \lambda \Phi)$, for $\lambda \in \mathbb{C}^*$. This action preserves the stability condition and the vanishing of Chern classes. Since the underlying Riemann surface is smooth, there is an induced algebraic action of $\mathbb{C}^*$ on the moduli space $\mathcal{M}_L(G^{\mathbb{R}})$ following from the construction of the moduli space. The next proposition provides a full characterization of the $\mathbb{C}^*$-fixed points in the moduli space $\mathcal{M}_L(G^{\mathbb{R}})$.

\begin{proposition}{\cite[Proposition~5.5]{BGG2006}}
  \label{prop:Cstar-fixed-characterization}
  A polystable $L$-twisted $G^\R$-Higgs bundle $(\calP, \Phi) \in \calM_L(G^\R)$ is a $\C^*$-fixed point if and only if there is a semisimple infinitesimal gauge transformation $h_0 \in H^0(\calP(i\frakh^\R))$ that induces a characteristic grading on the adjoint bundle
  \begin{align*}
    \calP(\frakg) &= \bigoplus_{k \in \Z} \calP(\frakh)_k \oplus \bigoplus_{k \in \Z} \calP(\frakm)_k, 
  \end{align*}
  and $\Phi \in H^0(\calP(\frakm)_{-2}\otimes K)$. 
\end{proposition}

\begin{remark}
  \hfill
  \begin{enumerate}
  \item We rescaled the $\C^*$-action in the above proposition to be $\lambda \cdot (\calP,\Phi) = (\calP, \lambda^{-2}\Phi), \lambda \in \C^*$ for convenience of notation, but it detects the same fixed points. 
  \item A Higgs bundle of the above form is also known as a \emph{Hodge bundle} when $L=K$ (see \cite[Lemma 4.1]{Simpson}.
  \end{enumerate}
\end{remark}

Let $\set{f,h,e}$ be a magical triple, $G^\R$ be the associated canonical real form, and $G_{\calC}^{\R,s}$ be the simple part of the Cayley real form. Let $C$ and $H$ be the complexification of the maximal compact subgroups of $G_{\calC}^{\R,s}$ and $G^\R$ respectively. 

\begin{proposition}{\cite[Proposition~7.20]{Cayley}}
  Let $\calP_C$ be a stable $C$-bundle. Then $(\calP_C \star \calP_{K^{1/2}}(H), f \, dz)$ is a $\C^*$-fixed point. 
\end{proposition}

 We now introduce the following:

\begin{definition}
  We call $(\calP_{H},\Phi)$ a \emph{Cayley uniformizing Higgs bundle} if it is the image of a $\C^*$-fixed point $(\calP_C, f_{\calC}\,dz^{m_{\calC}+1}) \in \calM_{K^{m_{\calC}+1}}(G_\calC^{\R,s})$ under the Cayley map $\Psi$ (see \ref{eq:cayley-map}). That is, 
  \begin{itemize}
  \item $\calP_{H} = \calP_{C} \star \calP_{K^{1/2}}(H)$, where $\calP_{K^{1/2}} = \Fr(K^{1/2})$ is the frame bundle of a theta characteristic. 
  \item $\Phi = (f + \tilde{f}_{\calC}) \, dz \in H^0(\calP_{H}(\frakm)\otimes K)$, where $\ad_f^{m_{\calC}}(\tilde{f}_{\calC})=f_{\calC}$. 
  \end{itemize}
\end{definition}

Since $(\calP_C, f_{\calC})$ is a $\C^*$-fixed point in $\calM_{K^{m_{\calC}+1}}(G_\calC^{\R,s})$, we may complete $f_{\calC}$ into an $\fraksl_2$-triple $\set{f_{\calC}, h_{\calC}, e_{\calC}}$ in $\frakg_0^s = (\frakg_{\calC}^{\R,s})^\C$, where $h_{\calC} \in i\frakc^\R$, and $e_{\calC} = -\rho(f_{\calC})$.

Set $\tilde{e}_{\calC} = -\rho(\tilde{f}_{\calC})$. Also recall that $\rho(e) = -f$. Then
\begin{equation}
  \label{eq:tilde-e}
  \tilde{e}_{\calC} = -\rho(k \ad_e^{m_{\calC}}(f_{\calC})) = (-1)^{m_{\calC}}k \ad_f^{m_{\calC}}(e_{\calC}) \in \frakg_{-2m_{\calC}}, 
\end{equation}
for some scalar multiple $k$.

\begin{lemma}
  \label{lem:raising-bracket-commute-on-centralizer}
  Let $x, Y, Z \in \frakg$ and $Z \in Z_{\frakg}(x)$, then for any $n \in \N$, 
  \begin{align*}
    [\ad_x^n(Y), Z] &= \ad_x^n([Y,Z]). 
  \end{align*}
\end{lemma}

\begin{proof}
  When $n=0$ the statement is true. Suppose the equality holds for $n=k$. By the Jacobi identity, we have
  \begin{align*}
    [\ad_x^{k+1}(Y), Z] &= [x, [\ad_x^k(Y), Z]] - [\ad_x^k(Y), [x, Z]] \\
                        &= [x, \ad_x^k([Y,Z])] = \ad_x^{k+1}([Y,Z]), 
  \end{align*}
  where we used that $[x,Z]=0$ since $Z \in Z_{\frakg}(x)$, and the induction hypothesis. 
\end{proof}

We now have the following characterization: 
\begin{proposition}
  \label{prop:Cayley-Cstar-fixed-point}
  A Cayley uniformizing Higgs bundle $(\calP_H, \Phi)$ is a $\C^*$-fixed point. In other words, if $(\calP_C, f_{\calC} \, dz^{m_{\calC}+1})$ is a $\C^*$-fixed point in $\calM_{K^{m_{\calC}+1}}(G_{\calC}^{\R,s})$, then $(\calP_C \star \calP_{K^{1/2}}(H), (f + \tilde{f}_{\calC}) \, dz)$ is a $\C^*$-fixed point, where $\ad_f^{m_{\calC}}(\tilde{f}_{\calC}) = f_{\calC}$. 
\end{proposition}

\begin{proof}
  Since $(\calP_C, f_{\calC})$ is a fixed point, Proposition \ref{prop:Cstar-fixed-characterization} tells us that there exists a semisimple element $h_{\calC} \in H^0(\calP_C(i\frakc^\R))$ such that $f_{\calC}$ is a degree $-2$ map for the grading induced by $\ad_{h_{\calC}}$. 

  \noindent Let
  \begin{align}
    \label{eq:h0-semisimple}
    h_0 := h + (m_{\calC}+1)h_{\calC}.
  \end{align}
  We claim that this gives the desired semisimple element in $H^0(\calP_H(i\frakh^\R))$. Indeed, since $\tau(h)=-h$ and $h_{\calC} \in i\frakc^\R$, we have $h_0 \in i\frakg^\R$. Then $\sigma(h)=h$ and $h_{\calC} \in i\frakc^\R \subset \frakh$ implies that $h_0 \in \frakh \cap i\frakg^\R = i\frakh^\R$.

  \noindent Note that $f_{\calC} \in \frakm_\calC$ implies that we may embed $f_{\calC}$ in $Z_{2m_{\calC}}$ by Equation (\ref{eq:ZmC-frakmC}). Hence, $\tilde{f}_{\calC} \in V_{2m_{\calC}} \subset \frakm$ and $f + \tilde{f}_{\calC} \in \frakm$. Now consider
  \begin{align*}
    [h_0, f + \tilde{f}_{\calC}] = -2f + 2m_{\calC}\tilde{f}_{\calC} + [(m_{\calC}+1)h_{\calC}, \tilde{f}_{\calC}]. 
  \end{align*}
  Note that $\tilde{f}_{\calC} \in V_{2m_{\calC}}$ is of weight $2m_{\calC}$, and $[h_{\calC},f]=0$ since $h_{\calC} \in \frakc$. Finally we may view $\tilde{f}_{\calC} = k \ad_e^{m_{\calC}}(f_{\calC})$ for some scalar $k$, then by Lemma \ref{lem:raising-bracket-commute-on-centralizer} we have $[h_{\calC}, \tilde{f}_{\calC}] = -2\tilde{f}_{\calC}$. Hence $f + \tilde{f}_{\calC}$ is of degree $-2$ with respect to the $\ad_{h_0}$-grading on $\frakg$ as required by Proposition \ref{prop:Cstar-fixed-characterization}. 
\end{proof}

\begin{remark} \hfill
  \begin{enumerate}
  \item Similar to the proof above, we also have $[h_0, e+\tilde{e}_\calC] = 2(e+\tilde{e}_\calC)$.
  \item Note that $f+\tilde{f}_{\calC}, h_0, e+\tilde{e}_{\calC}$ need not form an $\fraksl_2$-triple. However $h_0$ will be useful to us because we understand how it acts on the $\ad_h$-weight spaces for $h$ in the magical triple. 
  \end{enumerate}
\end{remark}

\begin{lemma}
  \label{lem:h0-gauge-action}
  Using the same notation as above, for $R^{h_0} := \exp(\log(R)h_0)$, we have
  \begin{align}
    \label{eq:h0-gauge-action}
    R^{h_0} \cdot (f + \tilde{f}_{\calC}) = R^{-2}(f + \tilde{f}_{\calC}), \quad R^{h_0}\cdot (e + \tilde{e}_{\calC}) = R^{2}(e + \tilde{e}_{\calC}). 
  \end{align}
\end{lemma}

\begin{proof}
  We will compute the adjoint action of $R^{h_0}$ on $f+\tilde{f}_\calC$. The action on $e+\tilde{e}_\calC$ can be checked similarly. We have
  \begin{align*}
    R^{h_0} \cdot (f + \tilde{f}_{\calC})      =& \Ad(\exp(\log(R)h_0))(f + \tilde{f}_\calC) \\
    =& \exp(\ad (\log(R)h_0))(f + \tilde{f}_\calC) \\
    =& \sum_{n=0}^\infty \frac{1}{n!}(\ad (\log(R)h_0))^n(f + \tilde{f}_\calC) \\
    =& (f + \tilde{f}_\calC) + [\log(R)h_0, (f+\tilde{f}_\calC)] + \frac{1}{2}[\log(R)h_0, [\log(R)h_0, (f+\tilde{f}_\calC)]] + \cdots \\
    =& (f + \tilde{f}_\calC) + (-2\log(R))(f + \tilde{f}_\calC) + \frac{(-2\log(R))^2}{2}(f+\tilde{f}_\calC) + \cdots \\
    =& \exp(-2\log(R))(f + \tilde{f}_\calC) \\
    =& R^{-2}(f + \tilde{f}_\calC). 
  \end{align*}
\end{proof}

\subsection{Slodowy slices at uniformizing Higgs bundles}\label{sec:Slodowy_uniform}

Let $p_0=(\calP_H, \Phi_0)$ be a Cayley uniformizing fixed point and let
\begin{align*}
  \hat{\frakm}_\calC = \set{ X \in \frakm_{\calC} \mid [h_{\calC}, X] = nX, n \ge 0 }. 
\end{align*}
Consider the base space
\begin{equation}\label{base_space}
  \calB_{p_0} := H^0(\calP_C(\hat{\frakm}_\calC) \otimes K^{m_{\calC}+1}) \oplus \sum_{i=1}^{r(e)} H^0(K^{l_i+1}).
\end{equation}
Then, for each $u \in \calB_{p_0}$ given by
\[u= (P_{\calC} \, dz^{m_{\calC}+1}, P_1 \, dz^{l_1+1}, \ldots, P_{r(e)} \, dz^{l_{r(e)}+1}),\]
we get a Higgs bundle $p_u=(\calP_H,\Phi_u)$, where for $\Phi_0 = (f + \tilde{f}_{\calC}) \, dz$, the Higgs field $\Phi_u$ is given by 
\begin{align}\label{Phi_u}
  \Phi_u = \Bigl(f + \tilde{f}_{\calC} + P_{\calC} X_{\calC} + \sum_{i=1}^{r(e)} P_i X_i \Bigr) \, dz, 
\end{align}
where $X_{i}$ is the highest weight vector in $S^{m_i}(\C^2)$. As a sum of irreducible representations of highest weight $2m_i$ for a $\Theta$-principal triple, we have
\begin{align*}
  W({2m_i}) \simeq Z_{2m_i} \otimes S^{m_i}(\C^2). 
\end{align*}

We now introduce the following: 

\begin{definition}\label{def:Slodowy_slice}
  For a Cayley uniformizing fixed point $p_0=(\calP_H, \Phi_0)$ and for the base space  $\calB_{p_0}$ as in (\ref{base_space}), the set of all such $(\calP_H,\Phi_u)$ is called a \emph{Slodowy slice} at the uniformizing fixed point $p_0$.
\end{definition}

\subsection{Harmonic reductions on Cayley uniformizing Higgs bundles}

Take the Cayley uniformizing Higgs bundle $p_0=(\mathcal{P}_H, \Phi_0)$ corresponding to the origin of the base space 
$$\mathcal{B}=H^0\left( \mathcal{P}_H(\mathfrak{m}) \otimes K\right).$$ 
From the uniformization theorem, there is a unique Riemannian metric $g_{\natural}$ of constant curvature -4 for the conformal class determined by the complex structure on the Riemann surface $X$. For a fixed parameter $R \in \mathbb{R}^+$, we consider, more generally, the unique metric $g_{\natural}/R^2$ with constant curvature $-4R^2$.  Let $\calQ_{\natural}(R)$ be the corresponding reduction of structure group of $\calP_{K^{1/2}}$ to $\U(1)$, for a choice of square root $K^{1/2}$ of the canonical line bundle $K$. 

Let $(\calP_C, f_\calC)$ be a point in the moduli space $\calM_{K^{m_{\calC}+1}}(G_{\calC}^{\R,s})$. By the Kobayashi--Hitchin correspondence, there is a harmonic reduction of structure group $\calQ_{\calC}$ of the bundle $\calP_C$ to the maximal compact subgroup $C^\R$.
This is inducing a natural metric on the bundle $\mathcal{P}_H$ as follows. 

\begin{definition}\label{def:harmonic_red}
  Let $\calQ_{0}(R) \subset \mathcal{P}_C \times_{X} \mathcal{P}_{K^{1/2}}$ be the natural reduction of structure group of the bundle $\mathcal{P}_C \times _{X} \mathcal{P}_{K^{1/2}}$. The \emph{natural reduction of structure group} $\calQ_{0}(R)$ of the bundle $\mathcal{P}_H=\mathcal{P}_C \star \mathcal{P}_{K^{1/2}}(H)$ is given by 
  \[\calQ_{0}(R)= \calQ_{\calC} \star \calQ_{\natural}(H) \subset \mathcal{P}_C \star \mathcal{P}_{K^{1/2}}(H) = \mathcal{P}_H.\]
\end{definition}

Abbreviating $\calQ_{0}(R=1)$ as $\calQ_{0}$, we have $\calQ_{0}(R)=R^{-h_0} \calQ_{\natural}$, for the semisimple element $h_0$ in the proof of Proposition \ref{prop:Cayley-Cstar-fixed-point}. 

For a Cayley uniformizing Higgs bundle $(\calP_H, \Phi_0)$, let $\calQ_0$ be the reduction of structure group to $H^\R$ from solving Hitchin's equations, and let $\rho_{\calQ_0}$ be the induced involution on $\ad \calP_H$ which leaves $\ad \calQ_0$ invariant. Let $D_{\calQ_0}$ be the Chern connection in $\ad \calQ_0$ with $D_{\calQ_0}^{(0,1)} = \delbar_{\calP_H}$. From the Kobayashi--Hitchin correspondence, we have 
\begin{align*}
  F(D_{\calQ_0}) + [(f + \tilde{f}_{\calC}) \, dz, -\rho_{\calQ_0}(f + \tilde{f}_{\calC}) \, d\bar{z}] = F(D_{\calQ_0}) + [(f + \tilde{f}_{\calC}) \, dz, (e + \tilde{e}_{\calC}) \, d\bar{z}] = 0. 
\end{align*}
Now let $\calQ_0(R)$ be the reduction of structure group associated to the involution
\begin{align*}
  \rho_{\calQ_0(R)}:=R^{-h_0} \circ \rho_{\calQ_0}, 
\end{align*}
where $R^{-h_0}:=\Ad(\exp(-\log(R)h_0))$ similar to Lemma \ref{lem:h0-gauge-action}. 

Analogously to \cite[Proposition 4.8]{Opers16} we then have the following:

\begin{proposition}
  The harmonic reduction on the Cayley uniformizing Higgs bundle $(\mathcal{P}_H, \Phi_0)$ with parameter $R$ is $\calQ_{0}(R)$. 
\end{proposition}

\begin{proof}

  \noindent Let $D_{\calQ_0(R)}$ be the Chern connection in $\ad \calQ_0(R)$ with $D_{\calQ_0(R)}^{(0,1)} = \delbar_{\calP_H}$. Then
  \begin{align*}
    &F(D_{\calQ_0(R)}) + R^2[(f + \tilde{f}_{\calC}) \, dz, -\rho_{\calQ_0(R)}(f + \tilde{f}_{\calC}) \, d\bar{z}] \\
    =& R^{-h_0} \circ F(D_{\calQ_0}) + R^2[(f + \tilde{f}_{\calC}) \, dz, R^{-h_0}(e + \tilde{e}_{\calC}) \, d\bar{z}] \\
    =& F(D_{\calQ_0}) + R^2[(f + \tilde{f}_{\calC}) \, dz,  R^{-2}(e + \tilde{e}_{\calC}) \, d\bar{z}], \quad \text{by Lemma \ref{lem:h0-gauge-action}} \\
    =& 0. 
  \end{align*}
  Note that $R^{-h_0} \circ F(D_{\calQ_0}) = F(D_{\calQ_0})$ because $e+\tilde{e}_{\calC}$ and $f+\tilde{f}_{\calC}$ are of weights $\pm 2$ with respect to $\ad_{h_0}$, and so their bracket must be of weight $0$ with respect to $\ad_{h_0}$. 
\end{proof}

\section{\texorpdfstring{$\Theta$}{Theta}-positivity and opers}

\subsection{A construction of \texorpdfstring{$G$-opers ($\Theta$-positive $G$-opers)}{G-opers (Theta-positive G-opers)}}
We next develop a general construction of a family of $G$-opers for a magical triple that is coming from Fuchsian uniformization. The way we construct this family is not via considering the $\Theta$-Slodowy functor of \cite[Definition 5.3]{CS}, but rather in the original way of Drinfeld--Sokolov \cite{DS1}, \cite{DS2}. The benefit from constructing a family in this way is two-fold. On the one hand, this allows us to relate our families of $G$-opers to corresponding explicit families of differential operators which are of independent interest. On the other hand, we show in the sequel that precisely these families appear as the scaling limits of Gaiotto. For a principal $\mathfrak{sl}_2$-triple, such families of $G$-opers have been constructed in \cite[Sections 4.8--4.9]{Opers16}.   

The construction has its origins in the work of Drinfeld--Sokolov \cite{DS1}, \cite{DS2}. However, Zucchini \cite[Section 3]{Zucchini} gives a description for a general $G$-bundle for a complex simple Lie group $G$ that is particularly useful for our cause.

\subsection{\texorpdfstring{$(G,P)$}{(G,P)}-opers and Slodowy slices}

Let $G$ be a connected complex semisimple Lie group. Let $\fraks=\ang{f,h,e}$ be an even $\fraksl_2$-triple. Recall that the Lie algebra has a characteristic decomposition $\frakg = \oplus_{k \in \Z} \frakg_k$ with respect to $\ad_h$. Let
\begin{align*}
  \frakp_\fraks = \sum_{k \ge 0} \frakg_k
\end{align*}
be the parabolic subalgebra associated to $\fraks$. It comes with a Levi decomposition for the Levi subalgebra $\frakl_\fraks = \frakg_0$ and the nilpotent subalgebra $\fraku_\fraks = \sum_{k > 0} \frakg_k$. Let $P_\fraks < G$ be the normalizer subgroup of $\fraku_\fraks$; its Lie algebra is $\frakp_\fraks$. 

\begin{definition}
  \label{defn:g-p-opers}
  Let $\fraks = \ang{f,h,e}$ be an even $\fraksl_2$-triple, and let $P=P_\fraks$ be the parabolic subgroup associated to $\fraks$. A \emph{$(G,P)$-oper associated to $\fraks$} on $X$ is a triple $(\calP_G, \calP_{P}, \omega)$ where $\calP_G$ is a holomorphic principal $G$-bundle on $X$, with a holomorphic reduction of structure group $\calP_{P}$, and $\omega$ is a holomorphic principal connection on $\calP_G$ whose second fundamental form
  \begin{align*}
    \Psi_\omega: \ad \calP_{P} \hookrightarrow \ad \calP_G \xrightarrow{\omega} \mathfrak{g} \twoheadrightarrow \mathfrak{g/p}
  \end{align*}
  takes values in the unique open dense orbit $P \cdot f$. 
\end{definition}

A morphism between two $(G,P)$-opers $(\calP_G,\calP_P,\omega)$ and $(\calQ_G,\calQ_P,\eta)$ is an \emph{isomorphism} of holomorphic principal $G$-bundles $\Phi: \calP_G \to \calQ_G$, whenever $\Phi|_{\calP_P}: \calP_P \simeq \calQ_P$ and $\Phi^*\eta=\omega$.


We next investigate $(G,P)$-opers in the case when $\mathfrak{s}$ is furthermore a $\Theta$-principal triple.

\begin{definition}\label{def:theta_pos_oper}
  Let $\fraks = \ang{f, h, e}$ be a $\Theta$-principal triple. Let $P_\Theta = P_\fraks$ be the parabolic subgroup associated to $\fraks$. Then we call a $(G,P_\Theta)$-oper in the sense above, a \emph{$\Theta$-positive oper}. 
\end{definition}

A \emph{Slodowy slice of $\Theta$-positive opers} at a Cayley uniformizing fixed point $p_0 = (\calP_H, \Phi_0)$, where $\Phi_0 = (f + \tilde{f}_{\calC}) \, dz$, is now determined as follows. We have $\calP_H = \calP_C \star \calP_{K^{1/2}} (H)$, where the star product comes from an embedding of a product of commuting subgroups
\begin{align*}
  m: C \times \C^* \hookrightarrow H.
\end{align*}
From the construction of the Cayley map (see \cite[Section~5.2]{Cayley}), we know that the embedded product subgroup $C \cdot \C^* = m(C \times \C^*)$ has Lie algebra $\frakc \oplus \ang{h} \subset \frakg_0$. Therefore $C \cdot \C^*$ normalizes $\fraku_\Theta = \oplus_{k > 0} \frakg_k$, and thus is a subgroup of $P_\Theta$.

Let $\calP_G = \calP_C\star \calP_{K^{1/2}}(G)$. Let $\calP_{P_\Theta} = \calP_C \star \calP_{K^{1/2}}(P_\Theta)$, which is well-defined by the above discussion. A \textit{Cayley uniformizing oper} is defined by the triple $(\calP_G, \calP_{P_\Theta}, \omega_0)$, where $\omega_0$ is the holomorphic connection on $\calP_G$ given by
\begin{align*}
  \omega_0 = \del^{\calQ_0} + \Phi_0. 
\end{align*}

The Cayley uniformizing oper is identified with the Cayley uniformizing Higgs bundle under the Non-abelian Hodge Correspondence. Since they carry the same data, we will abuse notation and also denote the uniformizing oper by $p_0$. We now set the following:

\begin{definition}\label{Slodowy_Theta_oper}
  Let $u \in \calB_{p_0}$ be a point as in (\ref{base_space}) and $\Phi_u$ as in (\ref{Phi_u}). The \textit{Slodowy slice of $\Theta$-positive opers} at the Cayley uniformizing oper is the set $\set{(\calP_G, \calP_{P_\Theta}, \omega_u)}$, where $\omega_u$ is the holomorphic connection on $\calP_G$ given by 
  \begin{align*}
    \omega_u = \del^{\calQ_0} + \Phi_u.   \end{align*}
  We will also consider the \textit{Slodowy slice of $\Theta$-positive opers at parameter} $\hbar \in \mathbb{C}^*$ to be given by the family $\set{(\calP_G, \calP_{P_\Theta}, \omega_u)}_{\hbar}$, where $\omega_u$ is the holomorphic connection on $\calP_G$ given by 
  $\omega_u = \del^{\calQ_0} + \hbar^{-1}\Phi_u$. 
\end{definition}

\begin{remark}  
  Note that by construction of the point $\Phi_u$ as in (\ref{Phi_u}) it follows directly that the second fundamental form takes values in the unique open dense orbit $P_{\Theta} \cdot f$, for $f$ the element in the $\Theta$-principal $\mathfrak{sl}_2$-triple. Thus, the family  $\set{(\calP_G, \calP_{P_\Theta}, \omega_u)}_{\hbar}$ above is indeed a family of  $\Theta$-positive opers in the sense of Definition \ref{def:theta_pos_oper}.
\end{remark}

In the rest of the present section, we make explicit the relationship between these families of $\Theta$-positive opers and certain families of generalized projective structures on a smooth surface. 

\subsection{Projective structures}

A (complex) \textit{projective structure} $Z$ on an oriented, connected smooth surface $S$ is a maximal atlas of smooth charts valued in $\CP^1$ with M\"obius transformations as transition maps. We write $Z$ to mean a smooth surface with projective structure. A smooth map $f: Z \to W$ is \textit{locally M\"obius} if around every $x \in Z$ and $f(x) \in W$, there are projective charts $(U, \phi)$ in $Z$ and $(V, \psi)$ in $W$ with $f(U) \subset V$, such that the coordinate representation $\hat{f} = \psi \circ f \circ \phi^{-1}$ is a M\"{o}bius transformation. 

Given a universal covering $p: \tilde{S} \to S$, we can lift a projective structure on $S$ to a projective structure $\tilde{Z}:=p^*Z$ on $\tilde{S}$. Then, a \textit{developing map} for $Z$ is a locally M\"{o}bius immersion $f: \tilde{Z} \to \CP^1$. Fix a projective structure $\tilde{Z}$ on the universal cover of $S$. A \textit{development-holonomy pair} on $S$ is a pair $(f,\rho)$, where $\rho: \pi_1(S) \to \PSL(2,\C)$ is a group homomorphism and $f: \tilde{Z} \to \CP^1$ is a locally M\"obius immersion which is $\pi_1(S)$-equivariant with respect to deck transformations on the source and the action given by $\rho$ on the target. 

\begin{proposition}[\cite{Dumas}]
  \label{prop:proj-str-dev-hol}
  There is a bijective correspondence between projective structures $Z$ on a surface $S$ up to marked isomorphism and development-holonomy pairs $(f,\rho)$ on $S$ up to an $\PSL(2,\C)$ action. 
\end{proposition}

Now we recall what is a $\PSL(2,\C)$-oper on a Riemann surface $X$. Let $B$ be a fixed Borel subgroup of $G=\PSL(2,\C)$. A \textit{$\PSL(2,\C)$-oper} on $X$ is a triple $(\calP_G,\calP_B,\omega)$, where $\calP_G$ is a holomorphic principal $G$-bundle, $\calP_B$ is a reduction of structure group to $B$, and $\omega$ is a holomorphic connection on $\calP_G$ with values in an open dense $B$-orbit $\mathcal{O} \subset \mathfrak{g}/\mathfrak{b}$.

The passage from a projective structure to a $\PSL(2,\C)$-oper can be given through the language of development-holonomy pairs. Given a projective structure on $X$, let $(f,\rho)$ be the development-holonomy pair associated to it by Proposition \ref{prop:proj-str-dev-hol}. Then, $\calP_G:=\tilde{X} \times_\rho \PSL(2,\C)$ is a holomorphic principal $G$-bundle on $X$. The developing map $f: \tilde{X} \to \CP^1$ gives a reduction of structure group, since $\CP^1 = G/B = \set{0 \subset L \subset \C^2}$ for $L$ one dimensional subspaces, and $f$ gives a global section of $\calP_G\times_G G/B$. The reduction of structure group to $B$ is given by the pullback bundle: 
\begin{align*}
  \calP_B := f^*(\calP_G\times_G B) = \set{ [(x,g)] \in \calP_G \mid f(x) = gB \in G/B }. 
\end{align*}
Let $\omega=d+df$, where $d$ is the Einstein--Hermitian connection (see \cite{AB}) on $\calP_G$. We have $df \not\equiv 0$ since by definition the developing map $f$ is locally injective. Hence $df$ takes values in the unique open $B$-orbit $\calO \subset \mathfrak{g}/\mathfrak{b}$ as we view the target space as the flag variety $G/B$. Clearly the triple $(\calP_G,\calP_B,\omega)$ defines a $\PSL(2,\C)$-oper. 

We explain the open dense $B$-orbit condition on the connection $\omega$ through the following example. Fix the Borel subgroup $B$ to be the upper triangular matrices. We have $\mathfrak{g}$ consisting of the $2 \times 2$ traceless matrices and $\mathfrak{b}$ consisting of the traceless upper triangular matrices. The Borel subgroup $B$ acts on $\mathfrak{g}/\mathfrak{b}$ with the adjoint action: 
\begin{equation}
  \begin{pmatrix}
    a & b \\
    0 & a^{-1} 
  \end{pmatrix} \cdot \begin{pmatrix}
    0 & 0 \\
    c & 0 
  \end{pmatrix} \cdot \begin{pmatrix}
    a^{-1} & -b \\
    0 & a
  \end{pmatrix} = \begin{pmatrix}
    a^{-1}bc & -b^2c \\
    a^{-2}c & -a^{-1}bc
  \end{pmatrix}. 
\end{equation}
Then $c \neq 0$ gives the unique open $B$-orbit in $\frakg/\frakb$. 

\subsection{Generalized opers}

Here we include a short description of the generalized $\sfB$-opers of \cite{BSY2020}. In some cases they give an explicit vector bundle formulation for the $\Theta$-positive opers, when the structure group is $\Sp(2n,\C)$, and the associated $\C^*$-fixed point agrees with the $\Theta$-principal or magical triple. 

Let $X$ be a compact Riemann surface of genus $g \ge 2$. Let $p: E \to X$ be a holomorphic vector bundle on $X$ equipped with a holomorphic nondegenerate bilinear from $\sfB: E \otimes E \to \calO_X$. A \textit{$\sfB$-connection} on $E$ is a holomorphic connection $D$ compatible with the bilinear form $\sfB$. That is, for any smooth sections $s, t \in \Omega^0(E)$, 
\begin{align*}
  \partial (\sfB(s,\, t))\,=\, \sfB(D(s),\, t) + \sfB(s,\, D(t)). 
\end{align*}
When $s, t$ are holomorphic sections, we have $\partial (\sfB(s,\, t))\,=\, d(\sfB(s,\, t))$. 

From now on we assume that $\sfB$ is reflexive, so that the following is well-defined. The \textit{orthogonal complement} of a holomorphic subbundle $F \subset E$ with respect to $\sfB$ is
\begin{align*}
  F^\perp := \set{w \in E \mid \sfB_{p(w)}(w,v)=0, \, \text{ for all  }\, v \in F_{p(w)}}. 
\end{align*}
A reflexive bilinear form over $\C$ is either alternating or symmetric, so without loss of generality we may assume that $\sfB$ is either a symplectic or a symmetric form. 

\begin{proposition}[\cite{BSY2020}]
  \label{prop:orthogonal-complement}
  The orthogonal complement $F^\perp$ satisfies the following properties.
  \begin{enumerate}
  \item $F^\perp \subset E$ is a holomorphic subbundle. 
  \item $F^\perp$ is canonically isomorphic to the dual bundle $(E/F)^*$.
  \item $(F^\perp)^\perp = F$. 
  \end{enumerate}
\end{proposition}

\begin{definition}
  \label{defn:b-filtration}
  A \textit{$\sfB$-filtration} of $E$ is a partial flag of holomorphic subbundles $\set{F_i}$: 
  \begin{align*}
    0 = F_0 \subsetneq F_1 \subsetneq F_2 \subsetneq \cdots \subsetneq F_n = E
  \end{align*}
  such that
  \begin{enumerate}
  \item the filtration length $n$ is odd when $\sfB$ is symplectic, and even when $\sfB$ is symmetric;     
  \item $F_i/F_{i-1}$ are holomorphic vector bundles; and
  \item $F_i^\perp = F_{n-i}$, for all $1 \le i \le n-1$. 
  \end{enumerate}
\end{definition}

In particular, the first half of a $\sfB$-filtration is fiberwise an isotropic flag of vector spaces. When $\sfB$ is symmetric and $n$ is odd, the middle subbundle is Lagrangian. 

\begin{definition}
  \label{defn:gen-b-opers}
  A \emph{generalized $\sfB$-oper} is a triple $(E,\set{F_i},D)$, where $\set{F_i}$ is a $\sfB$-filtration and $D$ is a $\sfB$-connection on $E$ that satisfies the conditions
  \begin{enumerate}
  \item (Griffiths transversality) $D(F_i) \subset F_{i+1} \otimes K$ for all $1 \le i \le n-1$; 
  \item (Nondegeneracy) for every $1 \le i \le n-1$, the second fundamental form of $D$ with respect to $F_i$ induces an isomorphism
    \begin{align*}
      S_D: F_i/F_{i-1} &\xrightarrow{\simeq} (F_{i+1}/F_i)\otimes K \\
      [s] &\mapsto [D(s)]. 
    \end{align*}
  \end{enumerate}
\end{definition}

\subsection{Generalized \texorpdfstring{$\sfB$}{B}-opers vs \texorpdfstring{$\Theta$}{Theta}-positive opers}

The following proposition follows directly from \cite[Theorem~4.1]{Yang2021} and provides a clear relationship between $\Theta$-positive opers in the case of the Lie group $G = \Sp(2n,\C)$ and generalized $\sfB$-opers.

\begin{proposition}
  Let $(E, \set{F_1,F_2}, D)$ be a generalized $\sfB$-oper, where
  \begin{itemize}
  \item $0 \to F_1 \to F_2$ for $E=F_2$ a rank $2n$ holomorphic vector bundle with a holomorphic symplectic form and $F_1$ a rank $n$ holomorphic subbundle
  \item $S_D: F_1 \xrightarrow{\simeq} F_2/F_1 \otimes K$ for $S_D$ the second fundamental form of $F_1$ for $D$. 
  \end{itemize}
  This is equivalent to the data of a $\Theta$-positive oper, for $G = \Sp(2n,\C)$ and $P_\Theta$ the parabolic subgroup associated to the $\Theta$-principal triple for $\Theta=\set{\alpha_n}$ the long restricted root in type $C$. 
\end{proposition}

\begin{remark} \hfill
  \begin{enumerate}
  \item It should be noted that the generalized $\sfB$-opers only describe a subset of the $\Theta$-positive opers, when the associated $\C^*$-fixed point arises from a $\Theta$-principal triple. In other words, the corresponding ``Cayley partner'' is not only stable as a Higgs bundle, but also stable as a vector bundle. Other $\Theta$-positive opers are associated to $\C^*$-fixed points taking the form of a nilpotent element of higher nilpotency degree than the $\Theta$-principal nilpotents.
    
  \item In general, for classical groups which have canonical representations, we expect the $\Theta$-positive opers to be vector bundles which admit a filtration given by the associated $\C^*$-fixed point, equipped with a compatible holomorphic connection whose second fundamental form is of maximal rank on the associated graded bundles.
    
  \item Finally, for the exceptional groups, we may obtain a vector bundle formulation of the $\Theta$-positive opers given a choice of faithful representation. 
  \end{enumerate}
\end{remark}

\section{Conformal limits}

We are now ready to state and prove our main theorem relating the two families of flat connections we introduced in the previous Sections, namely the 2-parameter family of flat connections coming from solutions of the Hitchin equations for the Cayley uniformizing Higgs bundles and the family of flat connections from $\Theta$-positive opers.

\begin{theorem}\label{main_theorem}
  Fix any point $u \in \mathcal{B}_{p_0}$ from \ref{base_space}. Let $(\mathcal{P}_H, \Phi_u)$ be a Higgs bundle in the Slodowy slice as in Definition \ref{def:Slodowy_slice} with $\calQ(R,u) \subset \mathcal{P}_H$ be the family of harmonic reductions solving the rescaled Hitchin equations (\ref{resc_Hit_eq:1}) and (\ref{resc_Hit_eq:2}). Fix a parameter $\hbar \in \mathbb{C}^*$ and consider the 2-parameter family of flat connections
  \begin{equation}\label{2-par_fam_flat_conn}
    \nabla_{R,\hbar,u}= \hbar^{-1}\Phi_u +D_{\calQ(R,u)}-\hbar R^2\rho_{\calQ(R,u)}(\Phi_u).
  \end{equation}
  Then, letting $R \to 0$, the flat connections $\nabla_{R,\hbar,u}$ converge to a flat connection
  \begin{equation}\label{conf_limit_conn}
    \nabla_{0,\hbar,u}=\hbar^{-1}\Phi_u +D_{\calQ_{0}}-\hbar\rho_{\calQ_{0}}(\Phi_0).
  \end{equation}
  The triple $(\calP_G, \calP_{P_\Theta}, \omega_u)$, where  $\calP_G$ is equipped with the holomorphic structure $\nabla^{0,1}_{0,\hbar,u}$ and $\omega_u=\nabla^{1,0}_{0,\hbar,u}$, is a $\Theta$-positive oper in a Slodowy slice at parameter $\hbar$ as in Definition \ref{Slodowy_Theta_oper}.
\end{theorem}

Before giving the proof of our main result above, we first consider the injectivity of a relevant elliptic operator at the Cayley uniformizing Higgs bundles.

\subsection{Injectivity at \texorpdfstring{$\C^*$}{C*}-fixed points}

At each point $u$ in the Slodowy slice, we have a real parameter family of Higgs bundles $(\calP_H,R\Phi_u)$ from Definition \ref{def:Slodowy_slice} with hermitian connections $D_{\calQ(R,u)}$ solving Hitchin's equations. Then there is an infinitesimal gauge transformation $\chi(R,u) \in \Omega^0(\calP_H(\frakh^\R))$ such that
\begin{align*}
  D_{\calQ(R,u)} &= \delbar_{\calP_H} + e^{-\chi} \circ \del_{\calP_H}^{\calQ_0(R)} \circ e^{\chi} \\
                 &= \delbar_{\calP_H} + e^{-\chi}R^{-h_0} \circ \del_{\calP_H}^{\calQ_0} \circ R^{h_0} e^{\chi}. 
\end{align*}
Note that the holomorphic structure remains the same since the holomorphic bundle $\calP_H$ remains unchanged in the Slodowy slice.

Consider now the elliptic operator
\begin{equation}\label{elliptic_operator}
  N_u(-,R): \Omega^0(\calP_H(\frakh^\R)) \to \Omega^2(\calP_H(\frakh^\R))
\end{equation}
given by 
\begin{equation*}
  N_u(\chi,R) = [\delbar_{\calP_H}, e^{-\chi} \circ \del_{\calP_H}^{\calQ_0} \circ e^{\chi}] - [\Phi_u, e^{-\chi} \circ \rho_{\calQ_0(R)} (\Phi_u)]. 
\end{equation*}

\begin{lemma}
  \label{lem:stability-lemma}
  Let $(\calP_H, \Phi)=\Psi(\calP_C, \psi, q) \in \calM(G^\R)$ be any Higgs bundle in a Cayley component. If it is stable as a $G$-Higgs bundle, then $\bbH^0(C^\bullet(\calP_C, \psi, q)) = 0$. 
\end{lemma}

\begin{proof}
  Since $(\calP_H, \Phi)$ is assumed to be stable as a $G$-Higgs bundle, Proposition \ref{prop:hypercohom-of-G-Higgs-bundles} gives us that $\bbH^0(C^\bullet(\calP_H,\Phi))=0$. Then \cite[Proposition~7.10]{Cayley} implies
  \begin{align*}
    \bbH^0(C^\bullet(\calP_C, \psi, q))=\bbH^0(C^\bullet(\calP_H,\Phi))=0.
  \end{align*}
\end{proof}

\begin{lemma}
  \label{lem:highest-weight-and-sigma-invariant}
  For a $\Theta$-principal triple $\langle f, h, e \rangle$ as was defined in Section \ref{sec:theta_pos_str}, we have  
  \begin{align*}
    \ker(\ad(e)) \cap \frakh = \frakc. 
  \end{align*}
\end{lemma}

\begin{proof}
  By definition we have
  \begin{align*}
    \ker(\ad(e)) = \frakc \oplus \bigoplus_{j} V_{2m_j}. 
  \end{align*}
  Recall that $\frakh = \frakg^\sigma$ is the fixed point set of $\sigma$, the magical involution. In particular $\sigma$ acts as $1$ on $\frakc$ and $-1$ on the highest weight spaces $V_{2m_j}$ by definition. 
\end{proof}

\begin{lemma}\label{lin_oper_bijective}
Let $\hat{\frakm}_\calC \subset \frakm_{\calC}$ be the subspace of non-negative $\ad_{h_{\calC}}$-weight spaces and let $u \in H^0(\calP_C(\hat{\frakm}_\calC) \otimes K^{m_{\calC}+1}) \oplus \sum_{i=1}^{r(e)} H^0(K^{l_i+1})$. Consider $\calL_u$, the linearization of $N_u(\chi,R)$ at $(0,0)$ with respect to $\chi$: 
  \begin{align*}
    &\calL_u: \Omega^0(\calP_H(\frakh^\R)) \to \Omega^2(\calP_H(\frakh^\R)) \\
    &\calL_u := D_\chi N_u|_{(0,0)} = \delbar_{\calP_H} \del_{\calP_H}^{\calQ_0} - [\Phi_u, [\rho_{\calQ_0}(\Phi_0), - ]]. 
  \end{align*}
  Then $\calL_u$ is a bijective linear operator. 
\end{lemma}

\begin{proof}
Let $(\calP_H, \Phi_0) = \Psi(\calP_C, f_\calC)$ be a Cayley uniformizing Higgs bundle in $\calM(G^\R)$, which is stable as a $G$-Higgs bundle. Then $(\calP_H, \Phi_0)=(\calP_C \star \calP_{K^{1/2}}(H), f + \tilde{f}_{\calC})$, where $\ad_f^{m_{\calC}}(\tilde{f}_{\calC})=f_{\calC}$, is a $\C^*$-fixed point in a Cayley component. 
We wish to show that for any $u$, the operator $\mathcal{L}_u$ is bijective.  Since $\calL_u$ deforms continuously to $\calL_0$, and the index of an operator is topologically invariant, it would be sufficient to show that $\calL_0$ is bijective. 

 \noindent Using the $L^2$-pairing given by the natural hermitian metric on $\calP_H(\frakh^\R)$, we have
  \begin{align*}
    (\calL_0\psi, \psi) &= (\delbar_{\calP_H} \del_{\calP_H}^{\calQ_0} \psi, \psi) - ([\Phi_0, [\rho_{\calQ_0}(\Phi_0), \psi]], \psi) \\
                        &= (\del_{\calP_H}^{\calQ_0} \psi, \del_{\calP_H}^{\calQ_0} \psi) + ([\rho_{\calQ_0}(\Phi_0), \psi],[\rho_{\calQ_0}(\Phi_0),\psi]) \\
                        &= (\psi, \calL_0 \psi). 
  \end{align*}
  Hence $\calL_0$ is self-adjoint, and thus it is bijective if and only if it is injective. It remains to show that $\calL_0$ is injective. 

 \noindent  We have $\calL_0\psi = 0$ if and only if
\[ \| \del_{\calP_H}^{\calQ_0} \psi \|^2 + \|[\rho_{\calQ_0}(\Phi_0), \psi]\|^2=\| \del_{\calP_H}^{\calQ_0} \psi \|^2 + \|[e + \tilde{e}_{\calC}, \psi]\|^2 = 0, \]
that is, if and only if 
\[ \del_{\calP_H}^{\calQ_0} \psi = 0, [e, \psi] = 0, \text{ and } [\tilde{e}_{\calC}, \psi] = 0.\]
Note that $\tilde{e}_{\calC} \in \frakg_{-2m_{\calC}}$ implies that $\ad_{\tilde{e}_{\calC}}$ is a degree $-2m_{\calC}$ map with respect to the $\ad_h$-grading, whereas $\ad_e$ is a degree $2$ map, hence the equivalence. By Lemma \ref{lem:highest-weight-and-sigma-invariant}, we have $[e, \psi] = 0$ implies that $\psi \in H^0(\calP_H(\frakc \cap \frakh^\R))$. Moreover, $\psi$ is real implies that
  \begin{align*}
    0 = \bar{\del_{\calP_H}^{\calQ_0} \psi} = \bar{\del_{\calP_C}^{\calQ_{\calC}} \psi} = \delbar_{\calP_C} \tau_\calC(\psi) = \delbar_{\calP_C} \psi.
  \end{align*}

\noindent  By Equation \ref{eq:tilde-e} and Lemma \ref{lem:raising-bracket-commute-on-centralizer} we get $[\tilde{e}_{\calC}, \psi] = 0$ implies $[e_{\calC}, \psi] = 0$, but then
  \begin{align*}
    [f_{\calC},\psi] = [\tau_\calC(e_{\calC}),\tau_\calC(\psi)] = \tau_\calC[e_{\calC},\psi] = 0.
  \end{align*}
  We have just shown that $\psi$ is an $f_{\calC}$-invariant holomorphic section. By the stability of $(\calP_H,\Phi_0)$ as a $G$-Higgs bundle and Lemma \ref{lem:stability-lemma}, we have $\bbH^0(C^\bullet(\calP_C,f_{\calC})) = 0$, which implies that $\psi = 0$. Hence $\calL_0$ is injective as required. 

  Next we consider the injectivity of $\calL_u$ in general. Let
  \begin{align*}
    \frakg = \bigoplus_{k \in \Z} \frakg'_k
  \end{align*}
  be the characteristic decomposition with respect to $\ad_{h_0}$. For each $k$, suppose $\psi \in \frakg_k' \cap \frakh^\R$, then
  \begin{align*}
    (\calL_u - \calL_0)\psi &= [\Phi_u-\Phi_0, [\rho_{\calQ_0}(\Phi_0), \psi]] \\
                            &= [P_{\calC}+\sum_{i=1}^{r(e)} P_i, [e + \tilde{e}_{\calC}, \psi]] \in \frakg'_{> k}
  \end{align*}
  has grading strictly larger than $k$. Compare this with
  \begin{align*}
    \calL_0\psi = \delbar_{\calP_H} \del_{\calP_H}^{\calQ_0} \psi + [f + \tilde{f}_{\calC}, [e + \tilde{e}_{\calC}, \psi]] \in \frakg'_k. 
  \end{align*}
  Since $\frakg'_{>k} \cap \frakg'_k = 0$, we see that the image of these two operators intersect trivially. 

\noindent Assuming $\calL_u\psi=0$, then
  \begin{align*}
    (\calL_u-\calL_0)\psi = \calL_0\psi \in \im(\calL_u-\calL_0) \cap \im(\calL_0) = 0. 
  \end{align*}
  By the injectivity of $\calL_0$, we get $\psi=0$, thus $\calL_u$ is injective.
\end{proof}

\subsection{Proof of Theorem \ref{main_theorem}}
The proof proceeds in a similar way to the one of Theorem 4.11 in \cite{Opers16} for the case of $G$-Higgs bundles in the Hitchin components. Namely, we show that that the harmonic reduction $\calQ(R,u)$ approaches $\calQ_{0}(R)$ as $R \to 0$. Remember from Definition \ref{def:harmonic_red} that 
\[\calQ_{0}(R)=R^{-\frac{h_0}{2}}\calQ_{0},\] 
for the semisimple element $h_0$ in Equation (\ref{eq:h0-semisimple}), where $\calQ_{0} \equiv \calQ_{0}(R=1)$.

\noindent Moreover, the harmonic reduction $\calQ(R,u)$ is described by
\[\calQ(R,u)=e^{-\frac{1}{2}\chi(R,u)}\calQ_{0}(R),\]
for a unique infinitesimal gauge transformation $\chi(R,u) \in \Omega^0(\calP_H(\frakh^\R))$. Then, for the nonlinear operator $N_u(-,R)$ in (\ref{elliptic_operator}), we have seen in Lemma \ref{lin_oper_bijective}
that its linearization is bijective. The real analytic version of the Implicit Function Theorem for Banach spaces now provides the existence of a real analytic (in $R$) section $\chi_0(R,u) \in \Omega^0(\calP_H(\frakh^\R))$, for $R \in [0,R_0)$ such that
\[N_u(\chi_0(R,u), R)=0,\]
and $\chi_0(0,u)=0$.

\noindent The real analytic map $\chi_0(R,u)$ thus expands in a Taylor series around $R=0$ and, as in the proof of \cite[Lemma 3.3]{Opers16}, plugging in this Taylor expansion into the equation $N_u(\chi(R,u), R)=0$ and expanding in powers of $R$, one sees that the first nonzero term in the expansion appears at order $R^4$. Thus, substituting 
\[\calQ(R,u)=e^{-\frac{1}{2}\chi_0(R,u)} \calQ_{0}(R)\]
we obtain
\begin{align*}
  \nabla_{R,\hbar,u} & = \hbar^{-1}\Phi_u +D_{\calQ(R,u)}-\hbar R^2\rho_{\calQ(R,u)}(\Phi_u)\\
                     & =\hbar^{-1}\Phi_u +e^{-\chi_0(R,u)} \circ D_{\calQ_{0}(R)} \circ e^{\chi_0(R,u)} - \hbar e^{-\chi_0(R,u)} \circ \rho_{\calQ_{0}}(\Phi_{R^2u}) \circ e^{\chi_0(R,u)}.
\end{align*}
In the limit $R \to 0$, we have $\chi_0(R,u) \to 0$, and so we get
\[\nabla_{0,\hbar,u}=\hbar^{-1}\Phi_u +D_{\calQ_{0}}-\hbar\rho_{\calQ_{0}}(\Phi_0).\]
The limiting connection $\nabla_{0,\hbar,u}$ above is a flat non-holomorphic connection, since it has a holomorphic and an anti-holomorphic part as follows 
\begin{align*}
  \nabla_{0,\hbar,u} & =\hbar^{-1}\Phi_u +D_{\calQ_{0}}-\hbar\rho_{\calQ_{0}}(\Phi_0)\\
                     & = \hbar^{-1}\Phi_u + \delbar_{\calP_H}+ \del_{\calP_H}^{\calQ_0}-\hbar\rho_{\calQ_{0}}(\Phi_0)\\
                     &= \underbrace{\del_{\calP_H}^{\calQ_{0}} + \hbar^{-1}\Phi_u}_{\nabla^{1,0}} +  \underbrace{\delbar_{\calP_H} -\hbar\rho_{\calQ_{0}}(\Phi_0)}_{\nabla^{0,1}}.
\end{align*}
Equipping the underlying principal bundle $\mathcal{P}_H$ with the holomorphic structure $\nabla^{0,1}$ above, we get a $\Theta$-positive oper $(\mathcal{P}_G, \mathcal{P}_{P_{\Theta}}, \omega)$ in the sense of Definition \ref{def:theta_pos_oper}.

\begin{remark}
  In \cite[Section 4]{Opers16}, a family of $G$-opers parameterized by the Hitchin base $\mathcal{B}= \bigoplus_i H^0\left( X, K_X^{m_i+1}\right)$ was constructed in terms of an explicit description of the transition functions defining the holomorphic $G$-bundles. This allowed for an explicit consideration of a smooth gauge transformation given by the matrix $M_{\hbar,z}= \text{exp}(\hbar(\partial_z \text{log} \lambda_{\natural})e)$ to compute directly that the limiting connection $\nabla_{0,\hbar,u}$ is gauge equivalent to a $G$-oper connection of the form $d+\hbar^{-1}\Phi_u$. In our more general case considered in this article, the principal bundles are constructed more abstractly, thus instead of modifying by a smooth gauge as in \cite{Opers16}, we equip the bundles with another holomorphic structure to show that the limit is a $\Theta$-positive oper.
\end{remark}




\bigskip
\noindent\small{\textsc{Department of Mathematics, University of Patras}\\
  Panepistimioupolis Patron, Patras 26504, Greece}\\
\emph{E-mail address}:  \texttt{gkydonakis@math.upatras.gr}

\bigskip
\noindent\small{\textsc{Kavli Institute for the Physics and Mathematics of the Universe}\\
  University of Tokyo, Kashiwa, Chiba 277-8583, Japan}\\
\emph{E-mail address}:  \texttt{mengxue.yang@ipmu.jp}


\begin{thebibliography}{9}

\bibitem{Aparicio}
  M.~Aparicio Arroyo, The geometry of $\text{SO}(p,q)$-Higgs bundles. Ph.D. thesis, Universidad de Salamanca, Consejo Superior de Investigaciones Cient\'{i}ficas, 2009.

\bibitem{AB}
  B.~Anchouche, I.~Biswas, Einstein--Hermitian connections on polystable principal bundles over a compact K\"{a}hler manifold. \textit{Amer. J. Math.} \textbf{123} (2001), 207--228. 
  
\bibitem{BR}
  I.~Biswas, S.~Ramanan, An infinitesimal study of the moduli of Hitchin pairs. \emph{J. Lond. Math. Soc., II.} \textbf{49} (1994), No. 2, 219--231. 


\bibitem{Cayley}
  S.~Bradlow, B.~Collier, O.~Garc\'{i}a-Prada, P.~Gothen, A.~Oliveira, A general Cayley correspondence and higher Teichm\"{u}ller spaces. \emph{Ann. Math.}, to appear (2024).

\bibitem{BGG2006}
  S.~B.~Bradlow, O.~Garc\'{i}a-Prada, P.~B.~Gothen, Maximal surface group representations in isometry groups of classical Hermitian symmetric spaces. \emph{Geom. Dedicata} \textbf{122} (2006), 185--213.

\bibitem{BSY2020}
  I.~Biswas, L.~Schaposnik, M.~Yang, Generalized $B$-opers. \emph{SIGMA} \textbf{16} (2020), 041, 28 pages. 

\bibitem{CS}
  B.~Collier, A.~Sanders, $(G,P)$-opers and global Slodowy slices. \emph{Adv. Math.} \textbf{377} (2021), Paper No. 107490, 43 pp.

\bibitem{CW}
  B.~Collier, R.~Wentworth, Conformal limits and the Bia\l{}ynicki-Birula stratification of the space of $\lambda$-connections. \emph{Adv. Math.} \textbf{350} (2019), 1193--1225. 
  
\bibitem{CoMc}
  D.~H.~Collingwood, W.~M.~McGovern, Nilpotent orbits in semisimple Lie algebras. Van Nostrand Reinhold Mathematics Series. Van Nostrand Reinhold Co., New York, (1993) xiv+186 pp.

\bibitem{DS1} 
  V.~G.~Drinfeld, V.~V.~Sokolov, Equations of Korteweg-deVries type and simple Lie algebras. \emph{Soviet Math. Doklady} \textbf{23} (1981), no. 3, 457--462.

\bibitem{DS2} V.~G.~Drinfeld, V.~V.~Sokolov, Lie algebras and equations of Korteweg-deVries type. \emph{J. Soviet Math.} \textbf{30} (1985), p. 1975--2035.

\bibitem{Dumas} E.,~Dumas, Complex Projective Structures. \textit{IRMA Lectures in Mathematics and Theoretical Physics}, Athanase Papadopoulos (ed.), 1st ed., 13:455--508. EMS Press (2009). 
  
\bibitem{Opers16} 
  O.~Dumitrescu, L.~Fredrickson, G.~Kydonakis, R.~Mazzeo, M.~Mulase, A.~Neitzke, From the Hitchin section to opers through nonabelian Hodge. \emph{J. Differential Geom.} \textbf{117} (2021), no. 2, 223--253.

\bibitem{Gaiotto}
  D.~Gaiotto, Opers and TBA. arXiv preprints, arXiv:1403.6137 (2014).

\bibitem{GMN}
  D.~Gaiotto, G.~W.~Moore, A.~Neitzke,
  Four-dimensional wall-crossing via three-dimensional field theory. \emph{Commun. Math. Phys.} \textbf{299} (2010), No. 1, 163--224. 

\bibitem{GGM2009}
  O.~Garc{\'i}a-Prada, P.~B.~Gothen, I.~Mundet i Riera. The Hitchin-Kobayashi correspondence, Higgs pairs and surface group representations. arXiv preprints, arXiv:0909.4487 (2009).
  
\bibitem{GW} O.~Guichard, A.~Wienhard, Positivity and higher Teichm\"{u}ller theory.
Mehrmann, Volker (ed.) et al., European congress of mathematics. Proceedings of the 7th ECM (7ECM) congress, Berlin, Germany, July 18--22, 2016.
  
\bibitem{Hit92}
  N.~Hitchin, Lie groups and Teichm\"{u}ller space. \emph{Topology} \textbf{31} (1992), no. 3, 449-473. 

\bibitem{Knapp}
  A.~W.~Knapp, Lie groups beyond an introduction. 2nd ed. \emph{Progress in Mathematics} (Boston, Mass.) 140. Boston, MA: Birkh\"{a}user, xviii + 812 p. (2002). 

\bibitem{Lusztig} 
  G.~Lusztig, Total positivity in reductive groups. \emph{Lie theory and geometry}, Progr. Math. vol.
  123, Birkh\"{a}user Boston, Boston, MA 1994, pp. 531--568.

\bibitem{Simpson}
  C.~T.~Simpson, Higgs bundles and local systems. \emph{Publ. Math. Inst. Hautes \'{E}tud. Sci.} \textbf{75} (1992), 5--95. 

\bibitem{Yang2021}
  M.~Yang, A Comparison of Generalized Opers and $(G, P)$-Opers. \emph{Indian J. Pure Appl. Math.} \textbf{53} (2022), no. 3, 760--773. 

\bibitem{Zucchini}
  R.~Zucchini, The Drinfel'd--Sokolov holomorphic bundle and classical $W$-algebras on Riemann surfaces. \emph{J. Geom. Phys.} \textbf{16} (1995), no. 3, 237--274. 


\end{thebibliography}
\end{document}